\documentclass[11pt]{amsart}
\usepackage{amsmath, amssymb, mathrsfs}
\usepackage{mathpazo}
\usepackage{color}
\usepackage{epic}
\usepackage{tikz}
\def\lbr{\left\{}
\def\rbr{\right\}}

\def\msf{{\mathscr F}}

\def\mA{{\mathbb  A}}

\def\me{{\mathbb  E}}

\def\mr{{\mathbb  R}}

\def\mn{{\mathbb  N}}

\def\mp{{\mathbb  P}}

\newcommand{\rf}[1]{(\ref{#1})}
\def\lnfrac#1#2{\raise.7ex \hbox{\Small $#1$}
  \kern-.15em/\kern-.15em  \lower.2ex \hbox{\Small $#2$}}

\theoremstyle{plain}
\newtheorem{theorem}{Theorem}[section]
\newtheorem*{theorem*}{Theorem}

\newtheorem*{cor*}{Corollary}
\newtheorem{lem}[theorem]{Lemma}
\newtheorem{prop}[theorem]{Proposition}
\newtheorem{defn}[theorem]{Definition}

\newtheorem{remark}[theorem]{Remark}

\numberwithin{equation}{section}

\usepackage{amsmath,amssymb,amscd,amsthm}

\usepackage{fullpage}

\usepackage[all,cmtip]{xy}
\usepackage[all]{xypic}

\usepackage{cite}

\usepackage{graphicx}
\usepackage{mathrsfs}

\newcommand{\smp}[1]{\left(#1 \right)}
\newcommand{\midp}[1]{\left[#1 \right]}
\newcommand{\bigp}[1]{\left\{#1 \right\}}

\newcommand{\R}{\mathbb{R}}

\newcommand{\N}{\mathbb{N}}

\newcommand{\E}{\mathcal{E}}

\DeclareMathOperator{\D}{\mathcal{D}}

\newcommand{\divg}[1]{\mathrm{div}{\; #1}}

\providecommand{\abs}[1]{\left| #1\right|}
\providecommand{\norm}[1]{\left\lVert#1\right\rVert}

\usepackage{esint}

\newcommand{\ip}[1]{\langle #1 \rangle}

\title{A probabilistic Harnack inequality and strict positivity of Stochastic Partial Differential equations}
\author{Zhenan Wang}
\address{}
\email{}

\begin{document}
\begin{abstract}
Under general conditions we show an \textit{a priori} probabilistic Harnack inequality for the non-negative solution of a stochastic partial differential equation of the following form
$$\partial_t u=\divg(\mA\nabla u)+f(t,x,u;\omega)+g_i(t,x,u;\omega)\dot{w}_t^i.$$
We will also show that the solution of the above equation will be almost surely strictly positive if the initial condition is non-negative and not identically vanishing.
\end{abstract}
\maketitle

\section{Introduction}
Stochastic partial differential equations (SPDEs) have been studied extensively during the last four decades. Fine properties for the solutions have always been a difficult topic. On the topic of positivity for the solution of linear SPDEs with multiplicative noise, it is well known since the beginning that the solution will remain non-negative if the initial condition is non-negative, see Krylov~\cite{Kry99} and Pardoux~\cite{Par07}. As for the strictly positivity of the solution, the question for stochastic heat equation is addressed by Carl Mueller \cite{M91} in $1991$. In their work in $1998$, \cite{TZ98}, Tessitore and Zabczyk have extended the result to a form that is more general. The strict positivity question can also be asked for non-linear SPDEs such as the ones studied in Debussche, De Moor and Hofmanova~\cite{DMH14} and Pardoux~\cite{Par75}. In particular, many examples of semi-linear SPDEs with measurable coefficients can be found in the survey monograph edited by Carmona and Rozovskii~\cite{Survey} and the answer to the strict positivity question for these equations is also unknown. The goal of this paper is to address such problem for a class of semi-linear SPDEs.

In the paper we consider the following type of SPDEs on $\R^n$:
\begin{equation}
\label{basiceqn}
\partial_t u=\divg(\mA\nabla u)+f(t,x,u;\omega)+g_i(t,x,u;\omega)\dot{w}_t^i,
\end{equation}
where $\{w^i\}$ is a sequence of independent standard Brownian motions on a filtered probability space $(\Omega, \msf_*, \mp)$ and  $g = \{g_i\}$ is an $\ell^2$-valued function such that for each fixed $x$ and an $\msf_* = \lbr\msf_t\rbr$-progressively measurable process $h$, the process $g(t,x,h_t;\omega)$ is also progressively measurable. We will show a probabilistic Harnack inequality for non-negative solutions of such equation and use the inequality to conclude that the solution stays strictly positive if the initial condition is non-negative and not identically vanishing. The probabilistic Harnack inequality is a local result, therefore we work on a domain $B$ in $\R^n$ along a time interval $I$ starting at $0$. The basic assumptions are as follows:

(1) uniform ellipticity: $\mA(t, x, u; \omega)$ is $\msf_*$-progressively measurable and uniformly elliptic on the space-time domain on which the solution lies, i.e., there is a positive constant $\iota$ such that
$$\iota \text{Id}\le \mA(t, x, u; \omega )\le \iota^{-1}\text{Id}, \quad  \forall  (t,x,u,\omega) \in I\times B \times\mr\times \Omega .$$

(2) linear growth near $\infty$ and linear decay near $0$: there exists a positive constant $ \Lambda $ such that
\[
\abs{f(t,x,u;\omega)}+ \abs{g (t,x,u;\omega)}_{\ell^2} \leq \Lambda|u|, \quad \forall (t,x,u;\omega) \in I \times B\times \R\times\Omega.
\]
We emphasize that no further conditions concerning the continuity $A, f$ or $g$ are imposed.

A function  $u = u(t,x;\omega)$ is said to be a (stochastically strong) solution of \eqref{basiceqn} on $I\times B$ if $u$ is almost surely a $L^\infty(I,L^2(B))$ process, lives in $L^2(\Omega\times I, W^{1,2}(B))$ and satisfies the corresponding partial differential equation (PDE) in the sense that
\[
\ip{u (t) , \varphi} = \ip{u (0) , \varphi}- \int_0^t \ip{ \mA\nabla u (s),  \nabla \varphi } \;ds  + \int_{0}^t \ip{ f(u (s) ), \varphi }\;ds +\int_0^t \ip{ g_i(u(s) )  , \varphi}\,dw_s^i
\]
for all $\varphi \in C_{c}^{\infty} (B)$. Here $\langle\cdot, \cdot\rangle$ denotes the standard inner product on $L^2(\mr^n)$. The probabilistic Harnack inequality is described in the following theorem.

\begin{theorem}[Probabilistic Harnack inequality]
\label{Main Theorem}
Let $U=I\times B$ be a bounded space-time rectangle and $u$ be a non-negative solution of \eqref{basiceqn} on $U$. Let $P$ and $Q$ be two bounded space time domains as shown in Figure \ref{Fig:1}, namely, $P$ is strictly after $Q$ in time, $Q$ is strictly after $0$ and both are contained in $U$.
Then for any $\epsilon>0$, we have a constant $\varGamma_0$ depending only on $n$, $\iota$, $\Lambda$ and the positions of $P$ and $Q$, such that for all $\varGamma>\varGamma_0$ and $a>0$,
\[
\mp\bigp{ \sup_{Q} u  > a,\quad  \varGamma  \inf_{P} u \leq a } \leq \epsilon.
\]
\end{theorem}
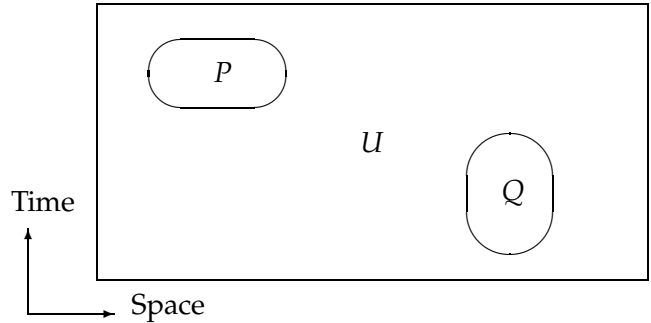
\begin{figure}[!htbp]
\setlength{\unitlength}{0.09in}
\centering
\begin{picture}(30,20)
\put(0,3){\framebox(32,16){$U$}}
\put(7,15){\oval(8,4){$ $}}
\put(24,8){\oval(5,7){$ $}}
\put(6.75,14.5){$P$}
\put(23.5,7.65){$Q$}
\put(-4,1){\vector(0,1){5}}
\put(-4,1){\vector(1,0){5}}
\put(-5,7){Time}\put(2,1){Space}
\end{picture}\label{Fig:1}
\caption{Relative positions of $P,Q$ and $U$.}
\end{figure}

Using this probabilistic Harnack inequality, we can show the strict positivity for \eqref{basiceqn}.
\begin{theorem}
\label{Positive}
Let $u$ be a solution of the SPDE \rf{basiceqn} on $\R^+\times\mr^n$ with a (deterministic) non-negative and not identically vanishing initial condition $u (0) =u_0\in L^2(\mr^n)$. Then almost surely, $u(t,x)>0$ for all $x\in\mr^n$ and $t>0$.
\end{theorem}

The methods we use in this article are drastically different from the conventional approaches used for positivity problems of SPDEs. We continue our work in \cite{HWW14} and combine ideas from Fabes and Garofalo~\cite{FG85} and Moser~\cite{Moser64}. Rather than relying on the solution kernel, we analyze the local behavior of the energy for the solution by a combination of PDE techniques and stochastic analysis. Our work can be viewed as a stochastic version of Moser's work including a stochastic version of the time-lagged bounded mean oscillation property, therefore our flexible method can potentially be further applied to other type of nonlinear SPDEs.

The paper is organized as follows. In {\sc Section 2}, we will present a four-step outline of the proof for {\sc Theorem \ref{Main Theorem}}, complete the proofs for the first and fourth steps in the outline, and prove {\sc Theorem \ref{Positive}}. In {\sc Sections 3, 4} and {\sc 5}, we will give proofs for the second step. In {\sc Sections 6}, we will give the proofs for the third step.

\section{Outlines of the proof}
In this section, we will first outline the proof for the deterministic parabolic Harnack inequality, and then develop a parallel process for {\sc Theorem \ref{Main Theorem}}.
In the following, we use $\norm{\cdot}_{p,D}$ to denote the $L^p$ norm on a domain $D$ in $\R^n$ or $\R^+\times\R^n$; thus $\norm{f}_{p,D}=(\int_Df^pdx)^{\frac{1}{p}}$ or $\norm{f}_{p,D}=(\int_Df^pdxdt)^{\frac{1}{p}}$, depending on the context.

With the same picture as in {\sc Figure \ref{Fig:1}}, Moser \cite{Moser64} established the deterministic Harnack inequality for parabolic equations as follows.
\begin{theorem}[Moser's parabolic Harnack inequality]\label{DHarnack}
Let $u$ be a non-negative solution of the parabolic equation
\begin{equation}\label{Deterministic}
\frac{\partial u}{\partial t}=\sum_{k,l=1}^n\frac{\partial}{\partial x_k}(a_{kl}(t,x)\frac{\partial u}{\partial x_l})
\end{equation}
in $U$ with $(a_{kl})$ uniformly elliptic.
For $P$ and $Q$ satisfying the same requirement as in {\sc Theorem \ref{Main Theorem}}, there exists a constant $C$ depending only on $(a_{kl})$ such that
$$\sup_{Q}u\leq C\inf_{P}u.$$
\end{theorem}

Moser's method establishes the inequality in the following four steps, as shown in {\sc Figure \ref{Fig:2}}:
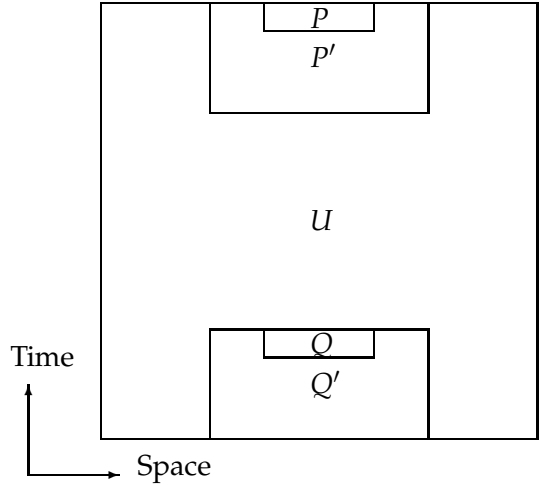
\begin{figure}[!htbp]
\setlength{\unitlength}{0.095in}
\centering
\begin{picture}(24,30)
\put(0,3){\framebox(24,24){$ $}}
\put(11.5,14.5){$U$}
\put(6,3){\framebox(12,6){$ $}}
\put(11.5,5.5){$Q'$}
\put(6,21){\framebox(12,6){$ $}}
\put(11.5,23.5){$P'$}
\put(9,25.5){\framebox(6,1.5){$ $}}
\put(11.5,25.65){$P$}
\put(9,7.5){\framebox(6,1.5){$ $}}
\put(11.5,7.75){$Q$}
\put(-4,1){\vector(0,1){5}}
\put(-4,1){\vector(1,0){5}}
\put(-5,7){Time}\put(2,1){Space}
\end{picture}
\caption{Relative positions of $P,Q,P',Q'$ and $U$.}
\label{Fig:2}
\end{figure}

\begin{enumerate}
\item We choose $P'$ and $Q'$ to be slightly larger than of $P$ and $Q$ respectively, $P'$ need to be strictly after $Q'$ while $Q'$ is allowed to touch time $0$, the sizes of the rectangles in {\sc Figure \ref{Fig:2}} are exaggerated.

\item For a fixed $\mu>0$, $1/(u+\mu)$ is a sub-solution of \eqref{Deterministic}. The De Giorgi iteration scheme shows that for all $p>0$, there exists $C_p>0$ such that $\sup_P\{1/(u+\mu)\}\leq C_p\norm{1/(u+\mu)}_{p,P'}$; at the same time $u+\mu$ is a solution of \eqref{Deterministic}, the same process gives $\sup_Q\{u+\mu\}\leq C_p\norm{u+\mu}_{p,Q'}$.

\item Now $-\log(u+\mu)$ is a sub-solution of an equation of the same type as \eqref{Deterministic}. This fact guarantees bounded mean oscillation (BMO) property in the parabolic sense for $\log(u+\mu)$. The parabolic John-Nirenberg inequality gives $\norm{1/(u+\mu)}_{p_0,P'}\norm{u+\mu}_{p_0,Q'}\leq K$ for some $p_0$ and $K$ independent of $\mu$.

\item Combining the results in the above two steps with $p=p_0$ in the second step, we have $\inf_Pu\geq C_{p_0}^{-2}K^{-1}\sup_Qu$ after letting $\mu\to 0$.
\end{enumerate}

To properly develop a stochastic version of Moser's method, we have to make two major difficulties. The first one is the lack of a definition of stochastic sub-solutions. It turns out that the na\"ive thought of simply changing the equality in the definition of the solutions to '$\leq$' is insufficient, as we need to describe the martingale property of the sub-solutions. We define a sub-solution as follows.
\begin{defn}
An almost surely bounded $L^2(B)$ process $u$ on $I$ living in $L^2(\Omega\times I, W^{1,2}(B))$ is a (stochastically strong) sub-solution of \eqref{basiceqn} on $I\times B$ if for all non-negative function $\phi\in C^\infty_c(B)$ and $s\leq t$,
\begin{enumerate}
\item $\displaystyle\langle u(t)-u(s),\phi\rangle\leq-\int_s^t\langle \mA\nabla u(\tau),\nabla\phi\rangle d\tau+\int_s^t\langle f(\tau),\phi\rangle d\tau+\int_s^t\langle g_i(\tau),\phi\rangle dw_\tau^i;$
\item the quadratic variation process of $\langle u,\phi\rangle$ at time $t$ equals to $\sum_i\int_0^t\langle g_i^2(\tau),\phi^2\rangle d\tau.$
\end{enumerate}
Here, $\langle\cdot,\cdot\rangle$ denotes the standard inner product on $L^2(B)$.
\end{defn}

The next difficulty is establishing stochastic version of inequalities in Moser's proof. Using a random variable $X$, say a certain norm of the solution $u$ of \eqref{basiceqn}, to bound another random variable $Y$, say another norm of $u$, with a non-random coefficient usually turns out impractical in the stochastic setting. Indeed, in our case, we cannot expect such kind of estimate to hold for norms of $u$ path-wise. Instead, to resemble the deterministic inequality $X\leq CY$, we use the tail probability of $Y$ to control the tail probability of $X$, namely,
$$\mp\bigp{X>a,CY\leq a}=o(1)\; \text{as}\; C\to\infty\;\text{for all}\;a>0.$$

At this point, we need to fix a few notations. For technical reasons, we use the maximum norm on $\R^n$, i.e., $|x|:=\max_i\{|x_i|\}$. We use the notation $B_r(x_0):=\{x\in\R^n||x-x_0|<r\}$. A $B_r$ without specifying the center will be understood as $B_r(0)$.
We also define $Q_r(t_0,x_0)$ as the space-time rectangular region $(t_0-r^2,t_0]\times B_r(x_0)$. A $Q_r$ without specifying the base point will be understood as $Q_r(1,0)$.

For a rectangular region $I\times B$, we define the following norms for all $p$ and $q$ positive,
$$\norm{h}_{p,q,I\times B}:=\norm{h}_{L^p(I,L^q(B))}=\smp{\int_I\norm{h}_{q,B}^pdt}^{1/p}.$$

In {\sc Section 3}, we will develop a local version of stochastic De Giorgi iteration from \cite{HWW14} and prove the following result.
\begin{prop}
\label{Local Max p large}
Let $u$ be a sub-solution of \eqref{basiceqn}
in $Q_1$. Then there exist $\varGamma(0)$ and $\delta(0)$ depending only on $n,\iota$ and $\Lambda$ such that for all $a>0$, $r\in(0,1]$, and $\varGamma\geq\varGamma(0)$,
$$\mp\bigp{ \norm{u}_{\infty, Q_{r/2} } >  a ,\; (r/2)^{-(n+1)/2} \norm{u}_{4,2 , Q_r} \leq a/\varGamma  } \leq \exp\bigp{-\varGamma^{\delta(0)}/r^2}.$$
The positions of $Q_{r/2}$ and $Q_r$ are shown in {\sc Figure \ref{Fig:3}}
\end{prop}
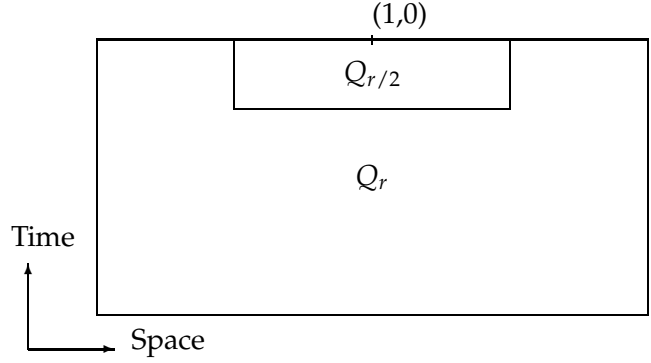
\begin{figure}[!htbp]
\setlength{\unitlength}{0.09in}
\centering
\begin{picture}(30,20)
\put(0,3){\framebox(32,16){$Q_r$}}
\put(8,15){\framebox(16,4){$Q_{r/2}$}}
\put(16,18.75){\line(0,1){0.5}}
\put(16,20){(1,0)}
\put(-4,1){\vector(0,1){5}}
\put(-4,1){\vector(1,0){5}}
\put(-5,7){Time}\put(2,1){Space}
\end{picture}
\caption{Relative positions of $Q_{r/2}$ and $Q_r$.}
\label{Fig:3}
\end{figure}

In {\sc Sections 4} and {\sc 5}, we will strengthen the above result into the following form, which will be used to prove the stochastic analogy of the second step in Moser's method.

\begin{prop}
\label{Local Max all p scaled} Let $u$ be a sub-solution to
\eqref{basiceqn} in $Q_1$. For every $2\geq p > 0$, there exist
$\varGamma(p), \delta(p)$ depending only on $n, \iota, \Lambda$ and $p$
such that for all $a>0$, $\varGamma \geq \varGamma(p)$ and $0<r<R\leq 1$,
\[
\mp\bigp{ \norm{ u}_{\infty, Q_r} >  a ,\;
(R-r)^{-(n+2)/p} \norm{u}_{p, Q_R} \leq a/\varGamma } \leq
\exp\bigp{-\varGamma^{\delta(p)}/R^2}.
\]
\end{prop}

To make our presentation for the analogy of the third step clearer, we will use the following notation from now on for any function $v>0$ and bounded measurable regions $D_1$ and $D_2$,
\[
\mathcal{F}[v, \alpha]_{D_1,D_2} := \smp{ \int_{D_1} v^{-\alpha}\;dxdt } \smp{ \int_{D_2}  v^{\alpha} \;dxdt }.
\]

In {\sc Sections 6}, we will provide a variant of the parabolic John-Nirenberg inequality in \cite[Theorem 1]{FG85}.
We then use this variant to prove the following reverse Cauchy-Schwarz inequality type statement.
\begin{prop}
\label{S JN}
Given $t\in(0,1)$, for every $\epsilon >0$, there exist constants $\alpha_{\epsilon}$ and $K_{\epsilon}$ depending only on $n, \iota, \Lambda,t$ and  $\epsilon$ such that $\forall \mu>0$ and any non-negative super-solution $u$ of \eqref{basiceqn} in $[0,2]\times B_1$.
\begin{equation}
\mp \bigp{ \mathcal{F}[u+\mu,\alpha_{\epsilon}]^{1/\alpha_{\epsilon}}_{D^+,D^-} > K_{\epsilon}  }  < \epsilon.
\end{equation}
Here, $D^+=(2-t^2,2)\times B_t$ and $D^-=(0,t^2)\times B_t$ as shown in {\sc Figure \ref{Fig:4}}.
\end{prop}
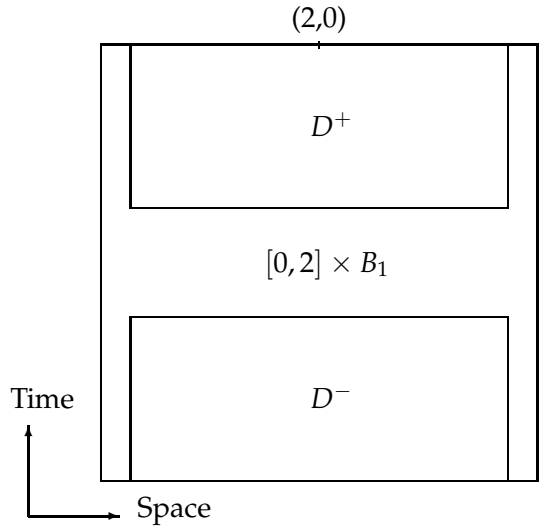
\begin{figure}[!htbp]
\setlength{\unitlength}{0.095in}
\centering
\begin{picture}(24,30)
\put(0,3){\framebox(24,24){$ $}}
\put(9,14.5){$[0,2]\times B_1$}
\put(1.6077,12){\line(1,0){20.7846}}
\put(1.6077,12){\line(0,-1){9}}
\put(22.3923,12){\line(0,-1){9}}
\put(1.6077,18){\line(1,0){20.7846}}
\put(1.6077,18){\line(0,1){9}}
\put(22.3923,18){\line(0,1){9}}
\put(11.5,22){$D^+$}
\put(11.5,7){$D^-$}
\put(-4,1){\vector(0,1){5}}
\put(-4,1){\vector(1,0){5}}
\put(12,26.8){\line(0,1){0.4}}
\put(10.5,28){(2,0)}
\put(-5,7){Time}\put(2,1){Space}
\end{picture}
\caption{Relative positions of $D^+,D^-$ and $[0,2]\times B_1$.}
\label{Fig:4}
\end{figure}

Now we prove the probabilistic Harnack inequality. For ease of reference, we restate it here.
\begin{theorem}
Let $U=I\times B$ be a bounded space-time rectangle and let $P$ and $Q$ be two bounded space time domains as shown in Figure \ref{Fig:1}, namely, $P$ is strictly after $Q$ in time, $Q$ is strictly after $0$ and both are contained in $U$.
Then for any $\epsilon>0$, we have a constant $\varGamma_0$ depending only on $n$, $\iota$, $\Lambda$ and the positions of $P$ and $Q$, such that for all $\varGamma>\varGamma_0$, $a>0$ and any non-negative solution $u$ of \eqref{basiceqn} on $U$.
\[
\mp\bigp{ \sup_{Q} u  > a,\quad  \varGamma  \inf_{P} u \leq a } \leq \epsilon.
\]
\end{theorem}
\begin{proof}
Without loss of generality, we first enlarge $P$ and $Q$ to be two space-time rectangular regions of the form $I_P\times B'$ and $I_Q\times B'$. We still require $I_P$ strictly after $I_Q$ and $I_Q$ strictly after $0$.

We will now consider two separate cases and prove the theorem for each of them.

\noindent{\bf Case I.} $I_P$ and $I_Q$ have the same length.

With proper scaling and translation, we can now assume $U$ contains $[0,2]\times B_1$ and pick up an $r\in(0,1)$ such that $P$ is contained in $Q_r(2,0)$ and $Q$ is contained in $Q_r(r,0)$.

We choose $R=\sqrt{r}>r$, then we have the inclusions $P\subset Q_r(2,0)\subset Q_R(2,0)\subset Q_1(2,0)$ and $Q\subset Q_r(r,0)\subset Q_R(r,0)\subset Q_1$, as shown in {\sc Figure \ref{Fig:5}}.
\begin{figure}[!htbp]
\setlength{\unitlength}{0.095in}
\centering
\begin{picture}(30,30)
\put(0,3){\framebox(24,24){$ $}}
\put(13,24){\framebox(5,2.5){$P$}}
\put(0.5,15.5){$Q_1(2,0)$}
\put(5.25,23.203125){\framebox(13.5,3.796875){$ $}}
\put(5.5,23.75){$Q_r(2,0)$}
\put(3,20.25){\framebox(18,6.75){$ $}}
\put(3.5,20.75){$Q_R(2,0)$}
\put(0.5,3.5){$Q_1$}
\put(5.25,5.953125){\framebox(13.5,3.796875){$ $}}
\put(13,7){\framebox(5,2.5){$Q$}}
\put(5.5,6.4){$Q_r(r,0)$}
\put(3,3){\framebox(18,6.75){$ $}}
\put(3.5,3.5){$Q_R(r,0)$}
\put(26,14.5){$[0,2]\times B_1$}
\put(0,15){\line(1,0){24}}
\put(-4,1){\vector(0,1){5}}
\put(-4,1){\vector(1,0){5}}
\put(12,26.8){\line(0,1){0.4}}
\put(12,9.55){\line(0,1){0.4}}
\put(10.5,10.75){$(r,0)$}
\put(10.5,28){$(2,0)$}
\put(-5,7){Time}\put(2,1){Space}
\end{picture}
\caption{Relative positions of the sets used in the proof.}
\label{Fig:5}
\end{figure}
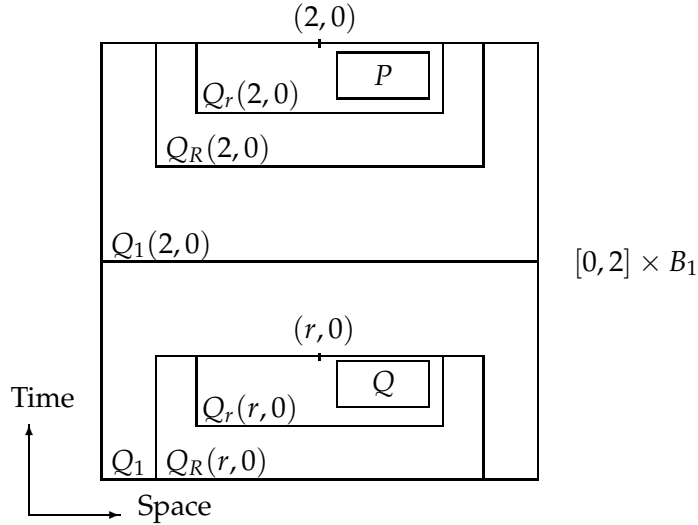

Fix any $\epsilon>0$ and let $\alpha_\epsilon$ and $K_\epsilon$ be the constants in {\sc Proposition \ref{S JN}} with $R$ in place of $t$ there.

For $\mu>0$, we write $$v_\mu=(u+\mu)^{-1},\;f^\mu(t,x,v_\mu;\omega):=f(t,x,u;\omega)v_\mu^2,\;g_i^\mu(t,x,v_\mu;\omega):=g_i(t,x,u;\omega)v_\mu^2.$$ By a direct calculation, $v_\mu(t+1,x)$ is a sub-solution of \eqref{basiceqn} on $Q_1(1,0)$ with $f$ and $g_i$ replaced by $f^\mu$ and $g_i^\mu$. We note here that $|f^\mu(v)|+\abs{g^\mu(v)}_{\ell^2}\leq\Lambda|v|$ still holds.

Applying {\sc Proposition \ref{Local Max all p scaled}} to $v_\mu(t+1,x)$ with $p=\alpha_\epsilon$ and $a$ replaced by $a^{-1}$, we can find a $\Gamma_\epsilon$ depending only on $n,\iota,\Lambda$ and $\epsilon$ such that for all $\Gamma\geq\Gamma_\epsilon$,
$$\mp\bigp{\sup_{Q_r(2,0)}v_\mu>a^{-1},\;\Gamma\norm{v_\mu}_{\alpha_\epsilon,Q_R(2,0)}\leq a^{-1}}\leq\exp\bigp{-\Gamma^{\delta(\alpha_\epsilon)}}.$$
This is equivalent to
\begin{equation}\label{p thm 1}
\mp\bigp{ \inf_{Q_{r}(2,0)} (u +\mu) < a ,\; \Gamma \norm{(u+\mu)^{-1}}_{\alpha_{\epsilon} , Q_R(2,0)} \leq a^{-1}   } \leq \exp\bigp{-\Gamma^{\delta(\alpha_\epsilon)}}.
\end{equation}

At the same time, we have an obvious inequality
\begin{align*}
\mp&\bigp{ \inf_{Q_{r}(2,0) } (u+\mu) < a ,\;\norm{u+\mu}_{\alpha_{\epsilon} , Q_R(r,0)}  \geq \Gamma K_{\epsilon} a} \\
& \leq \mp\bigp{\mathcal{F}[u+\mu, \alpha_{\epsilon}]^{1/\alpha_{\epsilon} }_{Q_R(2,0),Q_R(r,0)}  > K_{\epsilon} }\\
& + \mp  \bigp{ \inf_{Q_{r}(2,0) } (u+\mu) < a ,\;\norm{u+\mu}_{\alpha_{\epsilon},Q_R(r,0)}  \geq \mathcal{F}[u+\mu, \alpha_{\epsilon}]^{1/\alpha_{\epsilon} }_{Q_R(2,0),Q_R(r,0)}  \Gamma  a } .
\end{align*}
The first term on the right hand side is bounded by $\epsilon$ by {\sc Proposition \ref{S JN}} after taking $t=r$; the second term is equivalent to the left hand side of \eqref{p thm 1}, thus it is bounded by $\epsilon$ if $\Gamma$ is sufficiently large. Therefore the last inequality gives,
\begin{equation}
\label{p thm 2}
\mp\bigp{ \inf_{Q_{r}(2,0) } (u+\mu) < a ,\; \ \norm{u+\mu}_{\alpha_{\epsilon} , Q_R(r,0)}  \geq \Gamma K_{\epsilon} a}  \leq 2\epsilon.
\end{equation}

Now we look at $Q_R(r,0)$. On this rectangular region, $\bar{v}_\mu:=u+\mu$ is a solution of \eqref{basiceqn} with $f$ and $g_i$ replaced by $\bar{f}^\mu(t,x,\bar{v}_\mu;\omega):=f(t,x,u;\omega)$ and $\bar{g}_i^\mu(t,x,\bar{v}_\mu;\omega):=g_i(t,x,u;\omega)$. Applying {\sc Proposition \ref{Local Max all p scaled}} again, we can get another constant $\Gamma_\epsilon'$ depending only on $n,\iota,\Lambda$ and $\epsilon$ such that for all $\Gamma\geq\Gamma_\epsilon'$ and $a>0$,
\begin{equation}
\label{p thm 3}
\mp\bigp{ \sup_{Q_{r}(r,0) } (u+\mu) > a ,\; \ \Gamma\norm{u+\mu}_{\alpha_{\epsilon} , Q_R(r,0)}  \leq a} \leq 2\epsilon.
\end{equation}
From \eqref{p thm 2}, \eqref{p thm 3} and the obvious inequality
\begin{align*}
\mp&\bigp{\sup_{Q_{r}(r,0) } (u+\mu)>a,\;\ \Gamma\inf_{Q_{r}(2,0)}(u+\mu)< a}\\
&\leq \mp\bigp{\sup_{Q_{r}(r,0) } (u+\mu)>a,\;\ \norm{u+\mu}_{\alpha_\epsilon,Q_R(r,0)}\leq a/\Gamma_\epsilon'}\\
&\quad+\mp\bigp{\inf_{Q_{r}(2,0)}(u+\mu)< a/\Gamma,\;\ \norm{u+\mu}_{\alpha_\epsilon,Q_R(r,0)}\geq a/\Gamma_\epsilon'},
\end{align*}
we will have
$$\mp\bigp{\sup_{Q_{r}(r,0) } (u+\mu)>a,\;\ \Gamma\inf_{Q_{r}(2,0)}(u+\mu)< a}\leq 4\epsilon$$
if we pick $\Gamma_\epsilon'$ sufficiently large first and then let $\Gamma/\Gamma_\epsilon'$ be sufficiently large.

The last inequality implies
$$\mp\bigp{\sup_{Q_{r}(r,0) } (u+\mu)\geq 2a,\;\ \Gamma\inf_{Q_{r}(2,0)}(u+\mu)< a}\leq 4\epsilon.$$
Take $\mu=1/m$ and let $m\to\infty$, we have from Fatou's lemma,
$$\mp\bigp{\sup_{Q_{r}(r,0) }u\geq 2a,\;\ \Gamma\inf_{Q_{r}(2,0)}u< a}\leq 4\epsilon.$$
This further leads to
$$\mp\bigp{\sup_{Q_{r}(r,0) }u> 4a,\;\ \Gamma\inf_{Q_{r}(2,0)}u\leq a/2}\leq 4\epsilon,$$
which implies the desired statement if we use $\epsilon/4$ in place of $\epsilon$.

\noindent{\bf Case II.} The length of $I_P$ is different from that of $I_Q$ .

Without loss of generality we assume $I_P$ is longer, we cover $P$ by finitely many $P_i$'s of the exact same shape as $Q$. Applying the result from Case I to $P_i$ and $Q$ for all $i$, we have when $\Gamma$ is large
$$\mp\bigp{\sup_{P_i}u> a,\;\ \Gamma\inf_{Q}u\leq a}\leq \epsilon.$$
Therefore we have
$$\mp\bigp{\sup_{P_i}u> a,\;\ \Gamma\inf_{Q}u\leq a}\leq C\epsilon,$$
where $C$ is the number of rectangles used to cover $P$.
\end{proof}

We now turn to the strict positivity result, which we restate here.
\begin{theorem}
Let $u$ be a solution of the SPDE \rf{basiceqn} on $\R^+\times\mr^n$ with a (deterministic) non-negative and not identically vanishing initial condition $u (0) =u_0\in L^2(\mr^n)$. Then for probability one, $u(t,x)$ is positive for all $x\in\mr^n$ and $t>0$.
\end{theorem}
\begin{proof}
We prove by contradiction. Suppose the strict positivity conclusion is false, then for some $t_0>0$, we have
\begin{equation}\label{contradiction_1}\mp\bigp{\forall x\in\mr^n, \forall t\in[\frac{t_0}{2},t_0]\bigg|\; u(t,x)>0}<1.\end{equation}

We will first prove $u$ is non-negative. This is a well known result and the method of proof is to calculate $\me\norm{u^-}_{2,\mr^n}^2$ as in Pardoux~\cite{Par07}. By formally applying It\^o's formula on $h(u)=\abs{u^-}^2$, we have,
\begin{align*}
d\Vert u^-(t)\Vert^2_{2,\mr^n}= &-2\ip{\nabla u^-(t), \mA \nabla  u^- (t) } dt+  2 \ip{g_i (u)  , u^- (t)}  dw^i_t\\
&+\left[\int_{\R^n} \lbr \abs{g (u(t))}^2  +2u^-(t) f (u(t))\rbr 1_{\{ u^- (t)>0 \}} dx\right]  dt.
\end{align*}
The justification for the application of It\^o's formula is the same as in \cite[Remark 2.3]{HWW14}.

Taking the expectation on both sides and noting the fact that $u_0$ is non-negative, we have
$$\me\norm{u^-(s)}_{2,\mr^n}^2\leq\me\int_0^s\left[\int_{\R^n} \lbr \abs{g (u(t))}^2  +2u^-(t) f (u(t))\rbr 1_{\{ u^- (t)>0 \}} dx\right]  dt.$$
Using the linear growth condition on $f$ and $g$ and Gronwall's inequality, we have $\me\norm{u^-(t)}_{2,\mr^n}^2=0$ for all $t$. This proves the non-negativity of $u$ after we recall from \cite[Theorem 1.2]{HWW14} that $u$ is continuous in both time and space after time $0$.

With this non-negativity result, \eqref{contradiction_1} can be rewritten as
\begin{equation}\label{contradiction_2}\mp\bigp{\exists x\in\mr^n, t\in[\frac{t_0}{2},t_0]\bigg|\; u(t,x)=0}>0.\end{equation}
Due to the continuity of $u$ after $t=0$, if we cover $[t_0/2,t_0]\times\mr^n$ by countably many copies of $[t_0/2,t_0]\times \{x||x|\leq 1\}$, there must be one of them, say $P$, satisfying
$$\mp\bigp{\inf_Pu=0}>0.$$
At the same time, 
since the deterministic initial condition $u_0$ is not identically vanishing, there must be a small time after $0$ where $\norm{u}_{2,\mr^n}$ stays positive (the time may vary among different $\omega$'s). Therefore if we cover $(0,t_0/3]\times\mr^n$ by countably many compact sets, there must be one of them, say $Q_0$, satisfying
$$\mp\bigp{\inf_Pu=0,\;\sup_{Q_0}u>0}>0.$$
We now further choose an $a>0$ and a space-time domain $Q\subset(0,5t_0/12]\times\mr^n$ slightly larger than $Q_0$ such that
$$\mp\bigp{\inf_Pu=0,\;\sup_Qu>a}>0.$$
However, {\sc Theorem \ref{Main Theorem}} shows that we can find a large $\varGamma$ such that,
$$\mp\bigp{\varGamma\inf_Pu\leq a,\;\sup_Qu>a}<\mp\bigp{\inf_Pu=0,\;\sup_Qu>a}.$$
This gives a contradiction since the event on the right hand side implies the event on the left hand side.
\end{proof}

\section{Local properties of the sub-solution}
In our previous work \cite{HWW14}, a stochastic variant of the classical De Giorgi's iteration has been developed for studying the global properties of solution of \eqref{basiceqn}. In this section, we will adapt this method to prove local properties for the solution. In other words, we will prove {\sc Proposition \ref{Local Max p large}}. We start with the following result.

\begin{prop}
\label{Local Max p large_pre}
Let $\epsilon$ be a constant in $(0,1]$ and $u$ be a sub-solution to
\begin{equation}\label{basiceqn with epsilon}
\partial_t u  = \divg(\mA \nabla u) + \epsilon^2f(t,x,u;\omega)+ \epsilon\smp{ g_i (t,x,u;\omega)  dw^i_{t} }
\end{equation}
in $Q_1$. Then there exist $\varGamma(0)$ and $\delta(0)$ depending only on $n,\iota$ and $\Lambda$ such that for all $a>0$ and $\varGamma\geq\varGamma(0)$,
\begin{equation}\label{Local Max p large_pre_eqn}
\mp\bigp{ \sup_{Q_{1/2} } u >  a ,\; (1/2)^{-(n+1)/2} \norm{u}_{4,2 , Q_1} \leq a/\varGamma   } \leq \exp\bigp{-\varGamma^{\delta(0)}/\epsilon^2}.
\end{equation}
\end{prop}

The proof of the proposition is a verbatim repetition of the proof for \cite[Proposition 3.3]{HWW14} with minor adjustments.


We proceed as in \cite{HWW14}. We write $b_k=2^{-1}+2^{-k-1}$ and pick $I_k:=[1-b_k^2,1]$, a sequence of time intervals shrinking from $[0,1]$ to $[3/4,1]$. We define a sequence of smooth non-negative cut-off functions $\varphi_k$ bounded by $1$ such that $\varphi_k$ is $1$ on $B_{b_k}$ and $0$ outside of $B_{b_{k-1}}$ for $k\geq 1$ and $\varphi_0\equiv1$ on $B_1$. We also require $\varphi_k$ to have a gradient globally bounded by $n2^{k+2}$. For each $a>0$, we
write $u_{k,a}=(u-a(1-2^{-k}))^+$ and let
$$U_{k,a}:=||u_{k,a}\varphi_k||^2_{4,2,I_k\times B_1}.$$

For simplicity we denote $f(t,x,u;\omega)$ and $g_i(t,x,u;\omega)$ by $f(u)$ and $g_i(u)$, respectively. Assume $n\geq 3$ for now. We have the following iterative inequality.
\begin{prop}
\label{S itr prop} There exist constants  $C_0 = C(n, \iota, \Lambda)$ and $\delta=\delta(n, \iota, \Lambda)$ such that for $a\geq1$
\begin{equation}\label{S itr prop ineq}
U_{k,a} \leq C_0^k a^{-2\delta} \smp{U_{k-1,a} + X^*_{k-1,a}} U_{k-1,a}^{\delta},
\end{equation}
where
\begin{equation}
\label{def X}
X_{k-1,a}^* = \epsilon\sup_{1-b_{k-1}^2\leq s\leq t \leq 1}\int_s^t \ip{g_i (u(\tau)) , u_{k,a}(\tau)\varphi_k^2 }\;dw_\tau^i.
\end{equation}
\end{prop}

\begin{proof}
During this proof, the constant $C$ is enlarged from line to line as we proceed.

We note that $\Vert u_{k,a} (t)\varphi_k\Vert_{2,B_{b_{k-1}}}^2=\Vert u_{k,a} (t)\varphi_k\Vert_{2,B_1}^2$. H\"older's inequality with the conjugate exponents $(n+1)/n$ and $n+1$ gives
\begin{equation}
\Vert u_{k,a} (t)\varphi_k\Vert_{2,B_{b_{k-1}}}^2  \leq \norm{ u_{k,a} (t)\varphi_k}^{2}_{2(n+1)/n,B_{b_{k-1}}} \cdot \abs{ \lbr u_{k,a} ( t) > 0 \rbr\cap B_{b_{k-1}}}^{1/(n+1)}.
\label{holder}
\end{equation}
Using Chebyshev's inequality, we have
\[
\abs{\lbr u_{k,a} ( t) > 0\rbr\cap B_{b_{k-1}} }= \abs{\lbr u_{k-1,a} ( t) > 2^{-k}a \rbr\cap B_{b_{k-1}}}\leq
\left(\frac{2^k}a\right)^2\norm{u_{k-1,a} (t)}^2_{2,B_{b_{k-1}}}.
\]
Noting $\norm{u_{k-1,a} (t)}^2_{2,B_{b_{k-1}}}\leq \norm{u_{k-1,a} (t)\varphi_{k-1}}^2_{2,B_1}$, squaring \rf{holder} and integrating with respect to $t$ on $I_k$ we have
\[
U_{k,a}^2 \leq \left(\frac{2^k}a\right)^{4/(n+1)}\int_{I_k} \norm{u_{k,a} (t)\varphi_k}^4_{2(n+1)/n,B_{b_{k-1}}} \norm{u_{k-1,a}(t)\varphi_{k-1} }^{4/(n+1)}_{2,B_1} \,dt.
\]
Applying H\"older's inequality again with the same conjugate exponents, we obtain
\begin{equation}
\label{r2 eq 1}
\begin{split}
U_{k,a} \leq \left(\frac{2^k}a\right)^{2/(n+1)}&\times \smp{ \int_{I_k} \norm{u_{k,a}(t)\varphi_k}^{4(n+1)/n}_{2(n+1)/n,B_{b_{k-1}}} dt }^{n/2(n+1)}  \\
&\times\smp{ \int_{I_k} \norm{u_{k-1,a}(t)\varphi_{k-1}}_{2,B_1}^4\;dt  }^{1/2(n+1)}.
\end{split}
\end{equation}
The second factor is $\Vert u_{k,a}\varphi_k\Vert_{4(n+1)/n, 2(n+1)/n, I_k\times B_{b_{k-1}}}^2$, and the $L^p_tL^{q}_x$ interpolation inequality \cite[Proof of Proposition 2.1]{HWW14} leads to
\begin{equation*}
\Vert u_{k,a}\varphi_k\Vert_{4(n+1)/n, 2(n+1)/n, I_k\times B_{b_{k-1}}}^2\le\sup_{t \in  I_{k}}\norm{u_{k,a} (t)\varphi_k}_{2,B_{b_{k-1}}}^2  +  \int_{I_{k}} \norm{ u_{k,a} (t)\varphi_k }_{2n/(n-2),B_{b_{k-1}}}^2 \; dt.
\end{equation*}
Applying the Sobolev inequality on $B_1$ to the second term on the right side of the above inequality and then substituting the result in \eqref{r2 eq 1}, we obtain
\begin{equation}
\label{itr p 1}
U_{k,a} \leq C\left(\frac{2^k}a\right)^{2/(n+1)} \midp{ \sup_{t \in  I_{k}}\norm{u_{k,a} (t)\varphi_k}_{2,B_{b_{k-1}}}^2  +  \int_{I_{k}} \norm{\nabla  (u_{k,a} (t)\varphi_k )}_{2,B_{b_{k-1}}}^2dt}  U_{k-1,a}^{1/(n+1)}.
\end{equation}
after noting the fact that the third factor on the right side of \eqref{r2 eq 1} is bounded by $U_{k-1,a}^{1/(n+1)}$.

We now try to bound the right-hand side of \eqref{itr p 1}. For the same reasoning as in \cite[Remark 2.3]{HWW14}, It\^o's formula can be applied to the composition $h_k ( u(t) ) := \vert u_{k,a}(t)\varphi_k\vert^2$,
\begin{equation}\label{ito}
\begin{split}
d\Vert u_{k,a}(t)\varphi_k\Vert^2_{2,B_{b_{k-1}}}&= -2\ip{\nabla u_{k,a}(t), \mA \nabla  (u_{k,a} (t)\varphi_k^2 )} dt+  2 \ip{\epsilon g_i (u)  , u_{k,a} (t)\varphi_k^2}  dw^i_t\\
&+\left[\int_{B_{b_{k-1}}} \lbr \abs{\epsilon g (u(t))\varphi_k}^2  +2u_{k,a}(t)\varphi_k^2\epsilon^2 f (u(t))\rbr 1_{\{ u_{k,a} (t)>0 \}} dx\right]  dt.
\end{split}
\end{equation}
For the first term on the right-hand side, by using the uniform ellipticity of $A$ and the bounds on $\varphi$ and $\nabla\varphi_k$, we have
\begin{align*}
\ip{\nabla u_{k,a},\mA\nabla(u_{k,a} (t)\varphi_k^2 )}&=\ip{\nabla u_{k,a},\mA\varphi_k^2\nabla u_{k,a} (t)}+2\ip{\nabla u_{k,a},\mA u_{k,a} (t)\varphi_k\nabla\varphi_k }\\
&\geq \iota\norm{\varphi_k \nabla u_{k,a}}^2_{2,B_{1-b_{k-1}}}-\frac{\iota}{2}\norm{\varphi_k \nabla u_{k,a}}^2_{2,B_{b_{k-1}}}-C\norm{u_{k,a}\nabla\varphi_k}^2_{2,B_{b_{k-1}}}\\
&\geq \frac{\iota}{2}\norm{\varphi_k \nabla u_{k,a}}^2_{2,B_{b_{k-1}}}-C^k\norm{u_{k,a}}^2_{2,B_{b_{k-1}}}\\
&\geq \frac{\iota}{2}\norm{\varphi_k \nabla u_{k,a}}^2_{2,B_{b_{k-1}}}-C^k\norm{u_{k-1,a}\varphi_{k-1}}^2_{2,B_1}.
\end{align*}
For the third term on the right-hand side of \eqref{ito},  we observe that  if $u_{k,a}> 0$, then
$0< a \leq 2^{k} u_{k-1,a}$ and $0<u \leq u_{k-1,a} + a \leq (1+ 2^{k}) u_{k-1,a}$. By the linear growth assumption on $f$ and $g$, the third term is bounded by $C^k\Vert u_{k-1,a}\varphi_{k-1}\Vert_{2,B_1}^2\, dt$. Now, integrating \rf{ito} from $t'$ to $t$ with $t' \in  I_{k-1}\setminus I_{k}$ and $t \in  I_{k}$ and applying Cauchy-Schwartz inequality on the integral of $\norm{u_{k-1,a}\varphi_{k-1}}^2_{2,B_1}$ gives
\begin{equation*}
\begin{split}
\Vert u_{k,a}(t)\varphi_k\Vert^2_{2,B_{b_{k-1}}} + &\frac{\iota}{2} \int_{t'}^t\norm{\varphi_k\nabla u_{k,a}(s)}_{2,B_{b_{k-1}}}^2ds\\& \leq  \Vert \varphi_ku_{k,a} (t')\Vert_{2,B_{b_{k-1}}}^2
+C^k U_{k-1,a} + 2\epsilon\int_{t'}^t\ip{g_i (u(s))  , u_{k,a}(s)\varphi_k^2}dw^i_s.
\end{split}
\end{equation*}
Taking supremum over $t \in I_{k} $, we have
\begin{equation}
\label{Itr l2 2}
\begin{split}
\sup_{ t\in I_k}\Vert u_{k,a} (t)\varphi_k\Vert_{2,B_{1-b_{k-1}}}^2  &+ \int_{t_0}^{1} \norm{\varphi_k \nabla u_{k,a} (s)}_{2,B_{1-b_{k-1}}}^2 \; ds\\& \leq   C  \Vert u_{k,a} (t')\varphi_k\Vert_{2,B_{1-b_{k-1}}}^2+C^k U_{k-1,a}
+ CX^*_{k-1,a}
\end{split}
\end{equation}
with $X^*_{k-1, \alpha}$ defined in \eqref{def X}. Integrating \eqref{Itr l2 2} on $I_{k-1}\setminus I_k$ with respect to $t'$, combining the resulting inequality, \eqref{itr p 1} and the fact that
\begin{align*}
\int_{I_{k}} \norm{\nabla  (u_{k,a} (t)\varphi_k )}_{2,B_{b_{k-1}}}^2dt&\leq C\int_{I_{k}} \norm{\varphi_k\nabla  u_{k,a} (t) }_{2,B_{b_{k-1}}}^2dt+C\int_{I_{k}} \norm{u_{k,a} (t)\nabla\varphi_k}_{2,B_{b_{k-1}}}^2dt\\
&\leq C\int_{I_{k}} \norm{\varphi_k\nabla  u_{k,a} (t) }_{2,B_{b_{k-1}}}^2dt+C^k\int_{I_{k}} \norm{u_{k-1,a} (t)}_{2,B_{b_{k-1}}}^2dt\\
&\leq C\int_{I_{k}} \norm{\varphi_k\nabla  u_{k,a} (t) }_{2,B_{b_{k-1}}}^2dt+C^kU_{k-1},
\end{align*}
we obtain the desired iterative inequality \eqref{S itr prop ineq}.
\end{proof}
\begin{remark}
When $n=1$ or $2$, due to the different form of Sobolev inequality in those dimensions, we need to substitute the $1/(n+1)$ in the proof by some positive number $\delta$ between $0$ and $1/3$ and adjust the conjugate coefficients in the inequalities accordingly.
\end{remark}
We are ready to proceed to the next step, namely, comparing $X_{k,a}^*$ and $U_{k,a}$.

Consider the continuous martingale for any fixed $a>0$,
\begin{equation*} \label{spx}
X_{k,t} :=\epsilon\int_{0}^t \ip{g_i (u(s)) , u_{k+1,a}(s)\varphi_{k+1}^2} \;dw^i_s,
\end{equation*}
and recall from \rf{def X} that $X_{k,a}^* = \sup_{1-b_k^2\le s\le t\le 1}(X_{k,t}-X_{k,s})$.

\begin{lem}\label{A-X}
There exists a constant $C = C(n, \iota,\Lambda)$ such that for all positive $\alpha$, $\beta$ and $k$,
\[
\mp\lbr X_{k,a}^*\ge\alpha\beta, \;  U_{k,a} \leq\beta\rbr \leq {C} \, \exp\{ -\alpha^2/(C^k\epsilon^2)\}.
\]
\end{lem}
\begin{proof} We use $\langle X_k\rangle$ to denote the quadratic variation process of $X_{k,t}$. If we can show that there is a constant $C$ such that
\begin{equation}
\label{A-Q}
\langle X_k\rangle_1 - \langle X_k\rangle_{1-b_k^2} \leq \epsilon^2C^k U_{k,a}^2,
\end{equation}
then
$$\lbr X^*_{k, a}\ge\alpha\beta, U_{k a}\le\beta\rbr\subset\lbr \sup_{1-b_k^2\le s\le t\le 1}(X_{k,t}-X_{k,s})\ge\alpha\beta,
\langle X_k\rangle_1 - \langle X_k\rangle_{1-b_k^2}\le \epsilon^2C^k\beta^2\rbr$$
and the desired estimate follows immediately from the fact that $X_{k,t}$ is a time-changed Brownian motion and the corresponding estimate for Brownian motion, see \cite[Lemma 3.1]{HWW14} for details.  To prove \rf{A-Q}, we start with
$$\langle X_k\rangle_1 - \langle X_k\rangle_{1-b_k^2} =\epsilon^2 \sum_{i\in\mn}\int_{I_k}\langle g_i(u), u_{k+1}\varphi_{k+1}^2\rangle^2\, ds,$$
which follows from the definition of $X_{k,t}$.
We observe that  if $u_{k+1,a}> 0$, then
$0< a \leq 2^{k+1} u_{k,a}$ and $0<u \leq u_{k,a} + a \leq (1+ 2^{k+1}) u_{k,a}$. By Minkowski's inequality (integral form) and the linear growth assumption on $g$ we have
\begin{equation*}
\label{A-X p1}
\sum_{i\in \N} \smp{  \int_{\R^n } g_i (u) \; u_{k+1,a}\varphi_{k+1}^2\; d x}^2  \leq  \smp{
\int_{\R^n} |g (u)|  u_{k+1,a}\varphi_{k+1}^2 \;d x}^2 \leq C^k \smp{\int_{\R^n}
u_{k,a}^2\varphi_k^2 \;d x }^2.
\end{equation*}
Integrating over the interval $I_k$ we obtain the desired inequality \rf{A-Q}.
\end{proof}

With {\sc Lemma \ref{A-X}}, we can now use a Borel-Cantalli argument to prove {\sc Proposition \ref{Local Max p large_pre}}.
\begin{proof}[Proof of {\sc Proposition \ref{Local Max p large_pre}}]
We start with the observation that $\bigp{\norm{u^+}_{\infty, Q_{1/2}}>a} \subset G_a^c$,
where $G_a = \bigp{\lim_{k \rightarrow \infty} U_{k,a} = 0}$. Consider the events $\mathcal{E}_{k} = \{U_{k,a} \leq (a/\Gamma)^2\gamma^k \}$ for a constant $\gamma<1$ to be determined later. Since $\norm{u}_{4,2,Q_1} = \sqrt{U_{0, a}}$, it suffices to prove
\[
\mp\bigp{ G_a^c\cap\mathcal{E}_0}\le \exp\{-\varGamma^\delta/\epsilon^2\}.
\]
It is clear that
\[
G_{a}^c \subset \bigcup_{k \geq 0} \mathcal{E}_{k}^c  \subset  \mathcal{E}_{0}^c \cup \midp{ \bigcup_{k \geq 1} \smp{ \mathcal{E}_{k}^c \cap \mathcal{E}_{k-1} } },
\]
which implies
\begin{equation}
\label{Itr Nor 1}
\mp\lbr G_{a}^c \cap \E_{0}\rbr \leq \sum_{k \geq 1} \mp \lbr\mathcal{E}_{k}^c \cap \mathcal{E}_{k-1} \rbr.
\end{equation}
We estimate the probability  $\mp\lbr\mathcal{E}_{k}^c \cap \mathcal{E}_{k-1}\rbr $.

Let $\alpha = (2C)^{k/2}\varGamma^\delta$ and $\beta =  a^2\gamma^{k-1}\varGamma^{-2}$ in {\sc Lemma} \ref{A-X}. If $ X^*_{k-1,a} \leq \alpha\beta$ and $U_{k-1,a} \leq \beta $, then by the iterative inequality \rf{S itr prop ineq} in {\sc Proposition \ref{S itr prop}} we have (after canceling $a^{2\delta}$!)
$$U_{k,a} \le\frac{C_0^k}{a^{2\delta}}(\beta + \alpha\beta)\beta^{\delta} = \frac{(C_0\gamma^{\delta})^k(1+(2C)^{k/2}\varGamma^\delta)}{\gamma^{1+\delta}\varGamma^{2\delta}}\cdot \gamma\beta\le\gamma\beta.$$
The last inequality holds if we choose $\gamma$ sufficiently small such that $(C_1\gamma^{\delta})^k(1+(2C)^{k/2}\varGamma^\delta)\le \varGamma^\delta$ for all $k\ge 1$ and $\varGamma\ge 1$ and then $\varGamma$ sufficiently large such that $\gamma^{1+\delta}\varGamma^\delta\ge 1$.

Now the above inequality implies that $\mathcal{E}_{k}^c \cap \mathcal{E}_{k-1 }  \subset \{ X^*_{k-1,a} >\alpha\beta, U_{k-1,a} \leq \beta \} $. Its probability is estimated by  {\sc Lemma} \ref{A-X} and we have
$$\mp \lbr \mathcal{E}_{k}^c \cap \mathcal{E}_{k-1} \rbr  \leq C \exp\bigp{-\alpha^2/(C^k\epsilon^2)} = C\exp\bigp{-2^{k}\varGamma^{2\delta}/\epsilon^2}.$$
Using this in \eqref{Itr Nor 1} we obtain, again for sufficiently large $\varGamma$,
\[
\mp\lbr G_{a}^c \cap \mathcal{E}_{0}^c\rbr \leq C\sum_{k=1}^{\infty} \exp\bigp{-2^{k} \varGamma^{2\delta}/\epsilon^2} \leq   \exp\bigp{- \varGamma^{\delta}/\epsilon^2 }.
\]
By lowering $\delta$ a little bit, we have proved \eqref{Local Max p large_pre_eqn} for all $a>0$.
\end{proof}

\begin{remark}
The statements proved so far in this section do not require the initial condition of $u$ to be deterministic.
\end{remark}

At this moment, we are just one step away from {\sc Proposition \ref{Local Max p large}}. For the ease of reference, we restate our goal here.
\begin{prop}
\label{Local Max 4 2}
Let $u$ be a sub-solution to \eqref{basiceqn}
in $Q_1$. Then there exist $\Gamma(0)$ and $\delta(0)$ depending only on $n,\iota$ and $\Lambda$ such that for all $a>0$, $r\in(0,1]$, and $\Gamma\geq\Gamma(0)$,
\begin{equation}\label{Ineq_infinity/4 2 baby case}
\mp\bigp{ \norm{u}_{\infty,Q_{r/2} }>  a ,\; \Gamma (r/2)^{-(n+1)/2} \norm{u}_{4,2 , Q_r} \leq a   } \leq \exp\bigp{-\Gamma^{\delta(0)}/r^2}.
\end{equation}
\end{prop}
\begin{proof}
For any $r\in(0,1]$, we write $u_r(t,x;\omega):=u(r^2t,rx;\omega)$. By a direct calculation, $u_r$ satisfies equation
$$\partial_t{u_r}(t,x)=\divg(A\nabla u_r(t,x))+r^2f(r^2t,rx,u_r(t,x);\omega)+rg_i(r^2t,rx,u_r(t,x);\omega)\dot{w}_t^i,$$
therefore \eqref{Local Max p large_pre_eqn} gives the desired result with supremum in place of $\norm{\cdot}_\infty$.
By considering $-u$ instead of $u$, we have our desired inequality.
\end{proof}
\begin{remark}
As can be observed from the proof, for a fixed $r\leq 1$, the inequality \eqref{Ineq_infinity/4 2 baby case} will hold as long as $u$ is a sub-solution of \eqref{basiceqn} on $Q_r$.
\end{remark}

\section{From $||\cdot||_{4,2,\cdot}$ to $||\cdot||_{p,\cdot}$}

We have controlled the tail of $\norm{u}_{\infty,Q_{r/2}}$ by the
distribution of $\norm{u}_{4,2,Q_r}$ in the last section. We now
improve the control by lowering the $4,2$-norm to any small $p$-norm in this
section. We will prove {\sc Proposition \ref{Local Max all p scaled}}, which we restate here.

\begin{prop}
\label{Local Max all p scaled_copy} Let $u$ be a sub-solution to
\eqref{basiceqn} in $Q_1$. For every $2\geq p > 0$, there exist constants
$\varGamma(p)$ and $\delta(p)$ depending only on $n, \iota, \Lambda$ and $p$
such that for all $a>0$, $\varGamma \geq \varGamma(p)$ and $0<r<R\leq 1$,
\[
\mp\bigp{ \norm{u}_{\infty,Q_r } >  a ,\;
(R-r)^{-(n+2)/p} \norm{u}_{p, Q_R} \leq a/\varGamma } \leq
\exp\bigp{-\varGamma^{\delta(p)}/R^2}.
\]
\end{prop}

In fact, we will prove a more general result.

\begin{prop}
\label{Local Max all p q scaled} Let $u$ be a sub-solution to
\eqref{basiceqn} in $Q_1$. For all $4\geq p > 0$ and $2\geq q>0$, there exist constants
$\varGamma(p,q)$ and $\delta(p,q)$ depending only on $n, \iota, \Lambda$ and the pair $(p,q)$
such that for all $a>0$, $\varGamma \geq \varGamma(p,q)$ and $0<r<R\leq 1$,
\[
\mp\bigp{ \norm{u}_{\infty,Q_r } >  a ,\;
(R-r)^{-(n/q+2/p)} \norm{u}_{p,q, Q_R} \leq a/\varGamma } \leq
\exp\bigp{-\varGamma^{\delta(p,q)}/R^2}.
\]
\end{prop}

We will prove the above proposition as following. We will first strengthen {\sc Proposition \ref{Local Max 4 2}} by a covering argument, and then introduce two lemmas whose proofs will be postponed to the next section. After these, we will use the strengthened result as a starting point and repeatedly apply the two lemmas to prove {\sc Proposition \ref{Local Max all p q scaled}}.
\bigskip

Our strengthen version of {\sc Proposition \ref{Local Max 4 2}} takes the following form.
\begin{prop}
\label{Local Max 4 2 scaled} Let $u$ be a sub-solution to
\eqref{basiceqn} in $Q_1$, then there exist $\varGamma'(0),
\delta'(0)$ depending on only $n,\iota$, and $\Lambda$ such that for every $a>0$, $\varGamma \geq \varGamma'(0)$ and
$0<r<R\leq 1$,
\begin{equation}\label{Ineq_infinity/4 2}
\mp\bigp{ \norm{u}_{\infty,Q_r } >  a ,\;
(R-r)^{-(n+1)/2} \norm{u}_{4,2, Q_R} \leq a/\varGamma } \leq
\exp\bigp{\varGamma^{\delta'(0)}/R^2}.
\end{equation}
\end{prop}

\begin{proof}

From \eqref{Ineq_infinity/4 2 baby case}, we have for all $0<R\leq 1$,
$$
\mp\bigp{\norm{u}_{\infty,Q_{R/2}}>a,\;{(R/2)}^{-(n+1)/2}\norm{u}_{4,2,Q_R}\leq
a/\varGamma}\leq \exp\bigp{-\varGamma^{\delta(0)}/R^2}.
$$

We write $\theta=r/R\in(0,1)$ and consider the case $\theta\leq1/2$ first. We note that $Q_r\subset Q_{R/2}$ in this scenario. The last inequality implies
\begin{equation}\label{scaling theta small}
\mp\bigp{\norm{u}_{\infty,Q_r}>a,\;2^{(n+1)/2}(R-r)^{-(n+1)/2}\norm{u}_{4,2,Q_R}\leq
a/\varGamma}\leq \exp\bigp{-\varGamma^{\delta(0)}/R^2}.
\end{equation}
This leads to \eqref{Ineq_infinity/4 2} for $r\leq\frac{1}{2}R$ with some $\delta'(0)<\delta(0)$ and $\Gamma'(0)>\Gamma(0)$.

For the case $\theta>1/2$, we have $Q_{(1-\theta)R}(t_0,x_0)\subset Q_R$
for every $(t_0,x_0)\in Q_{\theta R}$, it follows from
{\sc Proposition \ref{Local Max 4 2}} applied
to $u(t-1+t_0,x-x_0)$,
\begin{align*}
&\mp\bigp{\norm{u}_{\infty,Q_{(1-\theta)R/2}(t_0,x_0)}>a,\;((1-\theta)R/2)^{-(n+1)/2}\norm{u}_{4,2,Q_{(1-\theta)R}(t_0,x_0)}\leq a/\varGamma}\\
&\leq\exp\bigp{-\varGamma^{\delta(0)}/((1-\theta)R)^2}.
\end{align*}
Since $Q_{(1-\theta)R}(t_0,x_0)\subset Q_R$, this inequality implies
$$\mp\bigp{\norm{u}_{\infty,Q_{(1-\theta)R/2}(t_0,x_0)}>a,\;((1-\theta)R/2)^{-(n+1)/2}\norm{u}_{4,2,Q_R}\leq a/\varGamma}\leq
\exp\bigp{-\varGamma^{\delta(0)}/((1-\theta)R)^2}.$$
Now there exists a constant $L$ depending only on the dimension $n$ such that we can choose
$\lceil L(1-\theta)^{-(n+2)}\rceil$ points $(t_i,x_i)$ in $Q_{\theta R}$ satisfying
$\displaystyle Q_{\theta R}\subset\bigcup_{i=1}^{\lceil L(1-\theta)^{-(n+2)}\rceil}Q_{(1-\theta)R/2}(t_i,x_i).$
This implies that the event $\bigp{\norm{u}_{\infty,Q_{\theta
R}}>a}$ is contained in $\displaystyle\bigcup_{i=1}^{\lceil L(1-\theta)^{-(n+2)}\rceil}\bigp{\norm{u}_{\infty,
Q_{(1-\theta)R/2}(t_i,x_i)}>a}.$ Therefore we obtain,

\begin{align*}
\mp&\bigp{\norm{u}_{\infty,Q_{\theta
R}}>a,\;((1-\theta)R)^{-(n+1)/2}\norm{u}_{4,2,Q_R}\leq
a/\varGamma}\\
&\leq\sum_{i=1}^{\lceil L(1-\theta)^{-(n+2)}\rceil}\mp\bigp{\norm{u}_{\infty,Q_{(1-\theta)R/2}(t_i,x_i)}>a,\;((1-\theta)R)^{-(n+1)/2}\norm{u}_{4,2,Q_R}\leq
a/\varGamma}\\
&\leq
2L(1-\theta)^{-(n+2)}\exp\bigp{-{2^{-(n+1)/2}\varGamma}^{\delta(0)}/((1-\theta)R)^2}\\
&=\exp\bigp{-\frac{\varGamma^{\delta(0)}}{R^2}\frac{2^{-(n+1)/2}}{(1-\theta)^2}+(n+2)\log\frac{1}{1-\theta}+\log(2L)}.
\end{align*}
Combining the last inequality, \eqref{scaling theta small}, and the
fact that
$$-\frac{\varGamma^{\delta(0)}}{R^2}\frac{2^{-(n+1)/2}}{(1-\theta)^2}+(n+2)\log\frac{1}{1-\theta}+\log(2L)<-2^{-(n+1)/2}\frac{\varGamma^{\delta(0)}}{R^2}, \forall \theta\in(\frac{1}{2},1),R\in(0,1]$$
when $\varGamma$ is large, we obtain {\sc Proposition \ref{Local Max 4 2
scaled}} by shrinking $\delta'(0)$ and enlarging $\varGamma'(0)$.
\end{proof}

\begin{figure}[!htbp]
\setlength{\unitlength}{0.09in}
\centering
\begin{picture}(40,25)
\put(0,3){\framebox(36,18){$Q_R$}}
\put(6,13){\framebox(24,8){$Q_r$}}
\put(6,20.5){\dashbox{0.125}(6,.5){}}
\put(12,20.5){\dashbox{0.125}(6,.5){}}
\put(18,20.5){\dashbox{0.125}(6,.5){}}
\put(24,20.5){\dashbox{0.125}(6,.5){}}
\put(6,20){\dashbox{0.125}(6,.5){}}
\put(12,20){\dashbox{0.125}(6,.5){}}
\put(18,20){\dashbox{0.125}(6,.5){}}
\put(24,20){\dashbox{0.125}(6,.5){}}
\put(6,19.5){\dashbox{0.125}(6,.5){}}
\put(12,19.5){\dashbox{0.125}(6,.5){}}
\put(18,19.5){\dashbox{0.125}(6,.5){}}
\put(24,19.5){\dashbox{0.125}(6,.5){}}
\put(6,19){\dashbox{0.125}(6,.5){}}
\put(12,19){\dashbox{0.125}(6,.5){}}
\put(18,19){\dashbox{0.125}(6,.5){}}
\put(24,19){\dashbox{0.125}(6,.5){}}
\put(6,18.5){\dashbox{0.125}(6,.5){}}
\put(12,18.5){\dashbox{0.125}(6,.5){}}
\put(18,18.5){\dashbox{0.125}(6,.5){}}
\put(24,18.5){\dashbox{0.125}(6,.5){}}
\put(6,18){\dashbox{0.125}(6,.5){}}
\put(12,18){\dashbox{0.125}(6,.5){}}
\put(18,18){\dashbox{0.125}(6,.5){}}
\put(24,18){\dashbox{0.125}(6,.5){}}
\put(6,17.5){\dashbox{0.125}(6,.5){}}
\put(12,17.5){\dashbox{0.125}(6,.5){}}
\put(18,17.5){\dashbox{0.125}(6,.5){}}
\put(24,17.5){\dashbox{0.125}(6,.5){}}
\put(6,17){\dashbox{0.125}(6,.5){}}
\put(12,17){\dashbox{0.125}(6,.5){}}
\put(18,17){\dashbox{0.125}(6,.5){}}
\put(24,17){\dashbox{0.125}(6,.5){}}
\put(6,16.5){\dashbox{0.125}(6,.5){}}
\put(12,16.5){\dashbox{0.125}(6,.5){}}
\put(18,16.5){\dashbox{0.125}(6,.5){}}
\put(24,16.5){\dashbox{0.125}(6,.5){}}
\put(6,16){\dashbox{0.125}(6,.5){}}
\put(12,16){\dashbox{0.125}(6,.5){}}
\put(18,16){\dashbox{0.125}(6,.5){}}
\put(24,16){\dashbox{0.125}(6,.5){}}
\put(6,15.5){\dashbox{0.125}(6,.5){}}
\put(12,15.5){\dashbox{0.125}(6,.5){}}
\put(18,15.5){\dashbox{0.125}(6,.5){}}
\put(24,15.5){\dashbox{0.125}(6,.5){}}
\put(6,15){\dashbox{0.125}(6,.5){}}
\put(12,15){\dashbox{0.125}(6,.5){}}
\put(18,15){\dashbox{0.125}(6,.5){}}
\put(24,15){\dashbox{0.125}(6,.5){}}
\put(6,14.5){\dashbox{0.125}(6,.5){}}
\put(12,14.5){\dashbox{0.125}(6,.5){}}
\put(18,14.5){\dashbox{0.125}(6,.5){}}
\put(24,14.5){\dashbox{0.125}(6,.5){}}
\put(6,14){\dashbox{0.125}(6,.5){}}
\put(12,14){\dashbox{0.125}(6,.5){}}
\put(18,14){\dashbox{0.125}(6,.5){}}
\put(24,14){\dashbox{0.125}(6,.5){}}
\put(6,13.5){\dashbox{0.125}(6,.5){}}
\put(12,13.5){\dashbox{0.125}(6,.5){}}
\put(18,13.5){\dashbox{0.125}(6,.5){}}
\put(24,13.5){\dashbox{0.125}(6,.5){}}
\put(6,13){\dashbox{0.125}(6,.5){}}
\put(12,13){\dashbox{0.125}(6,.5){}}
\put(18,13){\dashbox{0.125}(6,.5){}}
\put(24,13){\dashbox{0.125}(6,.5){}}
\end{picture}
\caption{An example of the covering used when $\theta=\frac{2}{3}$}
\label{Fig:6}
\end{figure}
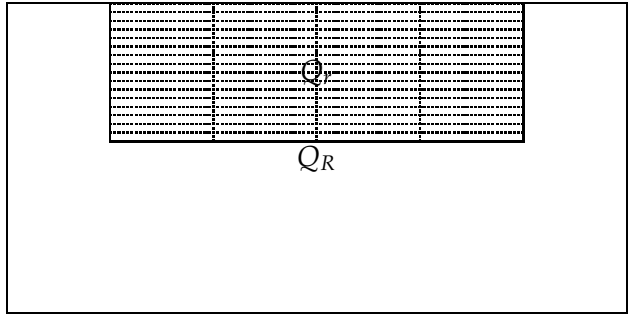

We now introduce the following two lemmas which will be proved later in {\sc Section 5}.
\begin{lem}[Exponent Reduction: Time]
\label{Exponent Reduction Time}
Let $u$ be a random function in $Q_1$. Suppose that for some $\alpha$ and $q$, $u\in L^\alpha([0,1],L^q(B_1))$ almost surely. Suppose further that there exist two constants $\delta(\alpha)>0$ and $
\varGamma(\alpha)$ such that for all $0<r<R\leq 1$, $a>
0$ and $\varGamma>\varGamma(\alpha)$,
\begin{equation}\label{Ineq_infinity/alpha q}
\mp\bigp{\norm{u}_{\infty,Q_r}>a,\;(R-r)^{-(n/q+2/\alpha)}\norm{u}_{\alpha,q,Q_R}\leq
a/\varGamma}\leq \exp\bigp{-\varGamma^{\delta(\alpha)}/R^2}.
\end{equation}
Then for any $\beta\in(\alpha/2,\alpha)$, there exist $
\delta(\beta)>0$ and $\varGamma(\beta)$ depending only on $n,\delta(\alpha)$ and $\varGamma(\alpha)$ such that
\begin{equation}\label{Ineq_infinity/beta q}
\mp\bigp{\norm{u}_{\infty,Q_r}>a,\; r^{-(n/q+2/\beta)}\norm{u}_{\beta,q,Q_{2r}}\leq
a/\varGamma}\leq \exp\bigp{-\varGamma^{\delta(\beta)}/4r^2},
\end{equation}
for all $\varGamma>\varGamma(\beta)$, $a> 0$ and $0<r\leq1/2$.
\end{lem}
\begin{lem}[Exponent Reduction: Space]
\label{Exponent Reduction Space} Let $u$ be a random function in $Q_1$. Suppose that for some $p$ and $\alpha$, $u\in L^p([0,1],L^\alpha(B_1))$ almost surely. Suppose further that there exist two constants $\delta(\alpha)>0$ and $
\varGamma(\alpha)$ such that for all $0<r<R\leq 1$, $a>
0$ and $\varGamma>\varGamma(\alpha)$,
\begin{equation}\label{Ineq_infinity/p alpha}
\mp\bigp{\norm{u}_{\infty,Q_r}>a,\;(R-r)^{-(n/\alpha+2/p)}\norm{u}_{p,\alpha,Q_R}\leq
a/\varGamma}\leq \exp\bigp{-\varGamma^{\delta(\alpha)}/R^2}.
\end{equation}
Then for any $\beta\in(\alpha/2,\alpha)$, there exist $
\delta(\beta)>0$ and $\varGamma(\beta)$ depending only on $n,\delta(\alpha)$ and $\varGamma(\alpha)$ such that
\begin{equation}\label{Ineq_infinity/p beta}
\mp\bigp{\norm{u}_{\infty,Q_r}>a,\; r^{-(n/\beta+2/p)}\norm{u}_{p,\beta,Q_{2r}}\leq
a/\varGamma}\leq \exp\bigp{-\varGamma^{\delta(\beta)}/4r^2},
\end{equation}
for all $\varGamma>\varGamma(\beta)$, $a> 0$ and $0<r\leq1/2$.
\end{lem}
\bigskip
To finish the proof of {\sc Proposition \ref{Local Max all p q scaled}}, we strengthen the conclusion of the two lemmas above, namely, \eqref{Ineq_infinity/beta q} and \eqref{Ineq_infinity/p beta} to the forms of \eqref{Ineq_infinity/alpha q} and \eqref{Ineq_infinity/p alpha} with the same covering argument as in the proof of {\sc Proposition \ref{Local Max 4 2 scaled}}. From here, {\sc Proposition \ref{Local Max 4 2 scaled}} can be viewed as a starting point and {\sc Proposition \ref{Local Max all p q scaled}} is obtained by repeatedly applying {\sc Lemma \ref{Exponent Reduction Time}} and {\sc Lemma \ref{Exponent Reduction Space}}.

\section{Proof of the Exponent Reduction Lemma}
In this section, we prove {\sc Lemma \ref{Exponent Reduction Time}} and {\sc Lemma \ref{Exponent Reduction Space}}. The two proofs are almost completely identical, they both come from combining an iterative method used by Fanghua Lin in \cite{LinBook} and the Borel-Cantelli type argument used in \cite[Proposition 3.3]{HWW14}; see \cite[Chapter 4]{LinBook} for a detailed exposition in the classical case.

We will only prove {\sc Lemma \ref{Exponent Reduction Time}} in detail and point out the differences for {\sc Lemma \ref{Exponent Reduction Space}}. We start with an auxiliary parameter $\tau\in(0,1)$ which will be determined later. Let $r_0=r$ and $r_{i+1}=r_i+r(1-\tau)\tau^i$ be a sequence of numbers increasing from $r$ to $2r$.

We will repeatedly use \eqref{Ineq_infinity/alpha q} on each pair of sets $(Q_{r_i},\,Q_{r_{i+1}})$ and then sum up the inequalities. However, the original form of \eqref{Ineq_infinity/alpha q} is not fit for estimation after summation since the variable $a$ in the inequality has to be a constant. To circumvent such issue, we introduce the following probabilistic lemma.

\begin{lem}\label{Prob lemma}
Assume $X,Y$ and $Z$ are three non-negative random variables with $Y\geq KZ$ for some $K>0$, suppose that there exist some $N_0$, $g$ and $\delta>0$ such that for all $b>0$ and $N>N_0$,
\begin{equation}\label{Prob lemma assumption}
\mp\bigp{X>b,\;YN\leq b}\leq \exp\bigp{-gN^\delta}.
\end{equation}
Then for all $b>0$ and $N$ such that $KN^2/(KN+1)>N_0$,
\begin{equation}\label{Prob lemma eqn}
\mp\bigp{X+Z>b,\;YN+Z\leq b}\leq \exp\bigp{-g(KN^2/(KN+1))^\delta}.
\end{equation}
\end{lem}
\begin{proof}
We show the following inclusion
$$\bigp{X+Z>b,\; YN+Z\leq b}\subseteq\bigp{X>KNb/(KN+1),\;YN\leq b}.$$
\eqref{Prob lemma eqn} follows immediately from this and \eqref{Prob lemma assumption}.

To prove the inclusion, we note that $Y\geq KZ\geq 0$ and $YN+Z\leq b$ imply $YN\leq b$ and $(KN+1)Z\leq b$. The second inequality and $X+Z>b$ then imply $X>b-Z\geq KNb/(KN+1)$.
\end{proof}

With {\sc Lemma \ref{Prob lemma}} in hand, we can now start the proof of {\sc Lemma \ref{Exponent Reduction Time}}.
\begin{proof}[Proof of {\sc Lemma \ref{Exponent Reduction Time}}]
We will first assume $r<1/2$. We write $\gamma=-1+\alpha/\beta<1$, $\lambda=n/q+2/\alpha$ and denote the volume of $B_1$ by $V$. We introduce an inequality which will play a key role in the proof of the lemma. From the $L^p_tL^q_x$ interpolation inequality and H\"older's inequality,  with some constant $C_{\alpha,\beta}$ we have for all $\epsilon>0$ and $l>0$,
\begin{equation}
\label{Interpolation Time}
\begin{split}
\norm{u}_{\alpha,q,Q_l}&\leq\norm{u}_{\infty,q,Q_l}^{1-\beta/\alpha}\norm{u}_{\beta,q,Q_l}^{\beta/\alpha}\\&\leq\epsilon\norm{u}_{\infty,q,Q_l}+C_{\alpha,\beta}\epsilon^{-\gamma}\norm{u}_{\beta,q,Q_l}\\&\leq\epsilon V^{1/q}l^{n/q}\norm{u}_{\infty,Q_l}+C_{\alpha,\beta}\epsilon^{-\gamma}\norm{u}_{\beta,q,Q_l}.
\end{split}
\end{equation}
For simplicity, we will use the following notations for this proof
$$F(l):=\norm{u}_{\infty,Q_l}, G(l):=\norm{u}_{\beta,q,Q_l},H(l):=\norm{u}_{\alpha,q,Q_l}.$$

Our formal strategy proving the lemma is to apply a Borel-Cantelli type argument. The argument works as following.
\begin{enumerate}
\item We create a sequence of sets $\bigp{\mathcal{S}_i}$ of the form $\bigp{2^{-i}F(r_i)+\sum_{j=1}^{i}c_jG(r_j)>a}$ for a sequence of positive numbers $\{c_i\}$ such that $\sum\limits_{i=1}^\infty c_i<\infty$. Note that $\mathcal{S}_0=\bigp{F(r_0)>a}$
\item We will give an estimate of the probability of the event $\mathcal{S}_0\bigcap\mathcal{S}_m^c$ that is uniform in $m$ by observing
$\mathcal{S}_0\bigcap\mathcal{S}_{m}^c\subset\bigcup_{i=0}^{m-1}\smp{\mathcal{S}_i\bigcap\mathcal{S}_{i+1}^c}$ and estimating each $\mp\bigp{\mathcal{S}_i\bigcap\mathcal{S}_{i+1}^c}$.
\item The estimates for $\mp\bigp{\mathcal{S}_i\bigcap\mathcal{S}_{i+1}^c}$ will be provided by first applying {\sc Lemma \ref{Prob lemma}} and \eqref{Ineq_infinity/alpha q} with carefully chosen parameters and then applying \eqref{Interpolation Time} with appropriate $\epsilon$.
\item After taking $m\to\infty$, the uniform estimate we have obtained in (2) will lead to an estimate for the left hand side of \eqref{Ineq_infinity/beta q}.
\end{enumerate}

We introduce another auxiliary parameter $\theta>0$ which will also be determined later and pick a sequence of numbers $M_i:=M_0/\tau^{i\theta}$ for some large $M_0>\varGamma(\alpha)$. This sequence will be used as a part of parameters for \eqref{Ineq_infinity/alpha q} and {\sc Lemma \ref{Prob lemma}}.

We now choose appropriate values for the sequence $\bigp{c_i}$. Applying \eqref{Ineq_infinity/alpha q} with $\varGamma=M_0$ on the pair of rectangles $(Q_{r_0},\;Q_{r_1})$ and \eqref{Interpolation Time} with $\epsilon=2^{-1}M_0^{-1}V^{-1/q}r_1^{-n/q}(r_1-r_0)^\lambda$, we have
\begin{equation}\label{S_0 cap S_1}
\begin{split}
&\quad\;\mp\bigp{F(r_0)>a,\;2^{-1}F(r_1)+\tilde{C}_{\alpha,\beta}M_0^{1+\gamma} r^{-(1+\gamma)\lambda}r_1^{\gamma n/q}G(r_1)\leq a}\\&\leq\mp\bigp{F(r_0)>a,\;M_0(r_1-r_0)^{-\lambda}H(r_1)\leq a}\\
&\leq\exp\bigp{-M_0^{\delta(\alpha)}/r_1^2}\leq\exp\bigp{-M_0^{\delta(\alpha)}/(4r^2)},
\end{split}
\end{equation}
where $\tilde{C}_{\alpha,\beta}=2^\gamma C_{\alpha,\beta}V^{\gamma/q}(1-\tau)^{-(1+\gamma)\lambda}$.

We now have a natural choice of $c_1=\tilde{C}_{\alpha,\beta}M_0^{1+\gamma} r^{-(1+\gamma)\lambda}r_1^{\gamma n/q}$. To find the appropriate value of other $c_i$'s, we look at the third step listed before. In step (3), we want to apply the interpolation inequality \eqref{Interpolation Time} on $M_i(r_{i+1}-r_i)^{-\lambda}H(r_{i+1})$ to get
$$M_i(r_{i+1}-r_i)^{-\lambda}H(r_{i+1})\leq 2^{-1}F(r_{i+1})+2^ic_iG(r_{i+1}),$$
therefore we choose $c_i=\tilde{C}_{\alpha,\beta}M_0^{1+\gamma} r^{-(1+\gamma)\lambda}r_i^{\gamma n/q}\kappa^{-i+1}$
where $\kappa=2\tau^{(1+\gamma)(\theta+\lambda)}$.

Now we define the sequence of sets
$$\mathcal{S}_i(M_0):=\bigp{2^{-i}F(r_i)+\tilde{C}_{\alpha,\beta}M_0^{1+\gamma}r^{-(1+\gamma)\lambda+\gamma n/q}\sum_{j=1}^{i}(r_j/r)^{\gamma n/q}\kappa^{-j+1}G(r_j)>a},$$
\eqref{S_0 cap S_1} now reads
$$\mp\bigp{\mathcal{S}_0(M_0)\bigcap\mathcal{S}_1^c(M_0)}\leq\exp\bigp{-M_0^{\delta(\alpha)}/(4r^2)}.$$

As we have stated before, we have
$$\mathcal{S}_0(M_0)\bigcap\mathcal{S}_{m}^c(M_0)\subset\bigcup_{i=0}^{m-1}\smp{\mathcal{S}_i(M_0)\bigcap\mathcal{S}_{i+1}^c(M_0)},$$
thus
$$\mp\bigp{\mathcal{S}_0(M_0)\bigcap\mathcal{S}_{m}^c(M_0)}\leq\sum_{i=0}^{m-1}\mp\bigp{\mathcal{S}_i(M_0)\bigcap\mathcal{S}_{i+1}^c(M_0)}.$$
We estimate the probability $\mp\bigp{\mathcal{S}_i(M_0)\bigcap\mathcal{S}_{i+1}^c(M_0)}$ for $i\geq 1$.

We start by choosing $Z=2^i\tilde{C}_{\alpha,\beta}M_0^{1+\gamma}r^{-(1+\gamma)\lambda+\gamma n/q}\sum_{j=1}^{i}(r_j/r)^{\gamma n/q}\kappa^{-j+1}G(r_j)$ with $X=F(r_{i+1})$ and $Y=H(r_{i+1})(r_{i+1}-r_i)^{-\lambda}$ in {\sc Lemma \ref{Prob lemma}}. We then choose the bounding coefficient $K=2^{-i}\kappa^{-1}\tilde{C}_{\alpha,\beta}^{-1}M_0^{-(1+\gamma)}r^{(1+\gamma)\lambda-\gamma n/q}[\sum_{j=1}^{i}(r_j/r)^{\gamma n/q}\kappa^{-j}]^{-1}r_{i+1}^{-2/\beta+2/\alpha}(r_{i+1}-r_i)^{-\lambda}$, as H\"older's inequality implies the estimate $H(r_{i+1})\geq G(r_{i+1})r_{i+1}^{-2/\beta+2/\alpha}$. At this moment, all the conditions of {\sc Lemma \ref{Prob lemma}} are satisfied with the help of \eqref{Ineq_infinity/alpha q} if we set $\delta=\delta(\alpha)$, $N_0=\varGamma(\alpha)$ and $g=1/(4r)^2$.
We want to use \eqref{Prob lemma eqn} with $b=2^ia$ and $N=M_i$. To do so, we show that $KN^2/(KN+1)>N_0$ when $M_0$ is large. Writing $\tilde{C}=2^{-2/\beta+2/\alpha}\tilde{C}_{\alpha,\beta}^{-1}(1-\tau)^{-\lambda}$ and recalling the definition of $\gamma$ and $\lambda$, we have
\begin{align*}
KN&=M_0\tau^{-i\theta}2^{-i}\kappa^{-1}\tilde{C}_{\alpha,\beta}^{-1}M_0^{-(1+\gamma)}r^{(1+\gamma)\lambda-\gamma n/q}[\sum_{j=1}^{i}(r_j/r)^{\gamma n/q}\kappa^{-j}]^{-1}r_{i+1}^{-2/\beta+2/\alpha}\tau^{-i\lambda}(1-\tau)^{-\lambda}r^{-\lambda}\\
&=M_0^{-\gamma}(2\tau^{\theta+\lambda})^{-i}\kappa^{-1}\tilde{C}[\sum_{j=1}^{i}(r_j/r)^{\gamma n/q}\kappa^{-j}]^{-1}(2r/r_{i+1})^{2/\beta-2/\alpha}.
\end{align*}
From here we first pick $\theta>\frac{2\gamma}{1-\gamma}\lambda$ so that $\frac{3+\gamma}{2}\theta+\lambda>(\theta+\lambda)(\gamma+1)$, then pick $\tau<1$ such that $2\tau^{\frac{3+\gamma}{2}\theta+\lambda}=1$. The choices of $\theta$ and $\tau$ also guarantee $\kappa>1$. Writing $C=2^{-\gamma n/q}(1-\kappa^{-1})\tilde{C}$, we estimate $KN$ as following,
$$KN\geq M^{-\gamma}(2\tau^{\theta+\lambda})^{-i}C(\sum_{j=0}^{i-1}\kappa^{-j})^{-1}(1-\kappa^{-1})^{-1}\geq M^{-\gamma}(2\tau^{\theta+\lambda})^{-i}C.$$
Therefore we have, when $M_0$ is large,
\begin{align*}
KN^2/(KN+1)&=N(KN)/(KN+1)\geq M_0\tau^{-i\theta}C[C+(2\tau^{\theta+\lambda})^iM_0^{\gamma}]^{-1}\\
&=M_0^{1-\gamma}\tau^{-i(1-\gamma)\theta/2}[M_0^{-\gamma}\tau^{i(1+\gamma)\theta/2}+C^{-1}(2\tau^{(3+\gamma)\theta/2+\gamma})^i]^{-1}\\
&\geq M_0^{1-\gamma}\tau^{-i(1-\gamma)\theta/2}(1+C^{-1})^{-1}\\&>N_0.
\end{align*}

From here, by \eqref{Prob lemma eqn} we obtain,
\begin{equation}\label{Exponent Reduction Time/BC Ineq}
\begin{split}
&\mp\bigg\{F(r_i)+2^i\tilde{C}_{\alpha,\beta}M_0^{1+\gamma}r^{-(1+\gamma)\lambda+\gamma n/q}\sum_{j=1}^{i}(r_j/r)^{\gamma n/q}\kappa^{-j+1}G(r_j)>2^ia,\\
&M_0\tau^{-i\theta}(r_i+1-r_i)^{-\lambda}H(r_{i+1})+2^i\tilde{C}_{\alpha,\beta}M_0^{1+\gamma}r^{-(1+\gamma)\lambda+\gamma n/q}\sum_{j=1}^{i}(r_j/r)^{\gamma n/q}\kappa^{-j+1}G(r_j)\leq2^ia\bigg\}\\
&\leq\exp\bigp{-(KN^2/(KN+1))^{\delta(\alpha)}/(4r^2)}\\
&\leq\exp\bigp{-(M_0^{1-\gamma}\tau^{-i(1-\gamma)\theta/2}(1+C^{-1})^{-1})^{\delta(\alpha)}/(4r^2)}\\
&\leq\exp\bigp{-(M_0\tau^{-i\theta})^{\delta(\alpha)(1-\gamma)/2}/(4r^2)}.
\end{split}
\end{equation}
It is worth noting that the estimates so far are uniform for all $i$ and $r\in(0,1/2)$ when $M_0$ is large.

We now use \eqref{Interpolation Time} with $\epsilon=2^{-1}M_0^{-1}V^{-1/q}r_{i+1}^{-n/q}(r_{i+1}-r_i)^\lambda$ and $l=r_{i+1}$. The last estimate gives,
$$\mp\bigp{\mathcal{S}_i(M_0)\bigcap\mathcal{S}_{i+1}^c(M_0)}\leq\exp\bigp{-(M_0\tau^{-i\theta})^{\delta(\alpha)(1-\gamma)/2}/(4r^2)},$$
which further leads to
$$
\mp\bigp{\mathcal{S}_0(M_0)\bigcap\mathcal{S}_{m}^c(M_0)}\leq\exp\bigp{-M_0^{\delta(\alpha)}/(4r^2)}+\sum_{i=1}^{m-1}\exp\bigp{-(M_0\tau^{-i\theta})^{\delta(\alpha)(1-\gamma)/2}/(4r^2)}.
$$

Recalling that $\gamma=-1+\alpha/\beta$ and $\lambda=2/\alpha+n/q$, we have the identity $(1+\gamma)\lambda-\gamma n/q=2/\beta+n/q$. This means the last inequality we have actually says,
\begin{align*}
\mp\bigg\{F(r)>a,&\;2^{-m}F(r_m)+\tilde{C}_{\alpha,\beta}M_0^{1+\gamma}r^{-(n/q+2/\beta)}\sum_{j=0}^{m-1}(r_j/r)^{\gamma n/q}\kappa^{-j}G(r_{j+1})\leq a\bigg\}\\
&\leq \exp\bigp{-M_0^{\delta(\alpha)}/(4r^2)}+\sum_{i=1}^{m-1}\exp\bigp{-(M_0\tau^{-i\theta})^{\delta(\alpha)(1-\gamma)/2}/(4r^2)},
\end{align*}
which implies
\begin{equation}\label{Exponent Reduction Time/One step to last}
\begin{split}
\mp\bigg\{F(r)>a,&\;2^{-m}F(2r)+2^{\gamma n/q}\tilde{C}_{\alpha,\beta}(1-\kappa)M_0^{1+\gamma}r^{-(n/q+2/\beta)}G(2r)< a\bigg\}\\
&\leq \exp\bigp{-M_0^{\delta(\alpha)}/(4r^2)}+\sum_{i=1}^{m-1}\exp\bigp{-(M_0\tau^{-i\theta})^{\delta(\alpha)(1-\gamma)/2}/(4r^2)}\\
&\leq \exp\bigp{-M_0^{\delta'}/(4r^2)}
\end{split}
\end{equation}
when $M_0$ is sufficiently large for $\delta'=(1-\gamma)\delta(\alpha)/4$.

Since $u\in L^\alpha([0,1],L^q(B_1))$ almost surely and $2r<1$, by choosing sufficiently large $a$ and $\varGamma$ in \eqref{Ineq_infinity/alpha q}, we obtain
$$\mp\bigp{\norm{u}_{\infty,Q_{2r}}<\infty}=1.$$
Taking $m\to\infty$, Fatou's Lemma and \eqref{Exponent Reduction Time/One step to last} gives,
\begin{equation*}
\mp\bigp{F(r)>a,\;2^{\gamma n/q}\tilde{C}_{\alpha,\beta}(1-\kappa)M_0^{1+\gamma}r^{-(n/q+2/\beta)}G(2r)< a}\leq \exp\bigp{-M_0^{\delta'}/(4r^2)}.
\end{equation*}
This implies
\begin{equation}\label{step before last reduction time}
\mp\bigp{F(r)>a,\;2^{1+\gamma n/q}\tilde{C}_{\alpha,\beta}(1-\kappa)M_0^{1+\gamma}r^{-(n/q+2/\beta)}G(2r)\leq a}\leq \exp\bigp{-M_0^{\delta'}/(4r^2)}.
\end{equation}
The inequality \eqref{step before last reduction time} implies \eqref{Ineq_infinity/beta q} for $\delta(\beta)=\delta'/(2(1+\gamma))$ and sufficiently large $\Gamma(\beta)$ for the case $0<r<1/2$.

It remains to deal with the case $r=1/2$. \eqref{step before last reduction time} with $r=1/2-\epsilon$ implies
\begin{equation*}\label{last step for reduction time}
\mp\bigp{F(1/2-\epsilon)>a,\;2^{1+\gamma n/q}\tilde{C}_{\alpha,\beta}(1-\kappa)M_0^{1+\gamma}(1/2-\epsilon)^{-(n/q+2/\beta)}G(1)< a}\leq \exp\bigp{-M_0^{\delta'}}.
\end{equation*}

Letting $\epsilon=1/k$ and $k\to\infty$, we get from Fatou's lemma,
$$\mp\bigp{F(1/2)>a,\;2^{1+\gamma n/q}\tilde{C}_{\alpha,\beta}(1-\kappa)M_0^{1+\gamma}(1/2)^{-(n/q+2/\beta)}G(1)< a}\leq \exp\bigp{-M_0^{\delta'}}.$$
The last inequality implies
$$\mp\bigp{F(1/2)>a,\;2^{2+\gamma n/q}\tilde{C}_{\alpha,\beta}(1-\kappa)M_0^{1+\gamma}(1/2)^{-(n/q+2/\beta)}G(1)\leq a}\leq \exp\bigp{-M_0^{\delta'}}.$$

This finishes the proof for {\sc Lemma \ref{Exponent Reduction Time}}.
\end{proof}

For {\sc Lemma \ref{Exponent Reduction Space}}, we need to use a different interpolation inequality
\begin{equation}
\label{Interpolation Space}
\begin{split}
\norm{u}_{p,\alpha,Q_l}&\leq\norm{u}_{p,\infty,Q_l}^{1-\beta/\alpha}\norm{u}_{p,\beta,Q_l}^{\beta/\alpha}\\&\leq\epsilon\norm{u}_{p,\infty,Q_l}+C_{\alpha,\beta}\epsilon^{-\gamma}\norm{u}_{p,\beta,Q_l}\leq\epsilon l^{2/p}\norm{u}_{\infty,Q_l}+C_{\alpha,\beta}\epsilon^{-\gamma}\norm{u}_{p,\beta,Q_l}.
\end{split}
\end{equation}

The Borel-Cantalli argument will be applied to the following sequence of sets
$$\bar{\mathcal{S}}_i(M_0):=\bigp{2^{-i}\norm{u}_{\infty,Q_{r_i}}+\bar{C}_{\alpha,\beta}M_0^{1+\gamma}r^{-(1+\gamma)\bar{\lambda}+2\gamma /p}\sum_{j=0}^{i-1}(r_j/r)^{2\gamma/p}\kappa^{-j}\norm{u}_{p,\beta,Q_{r_{j+1}}}>a}$$
with $\bar{C}_{\alpha,\beta}=2^\gamma C_{\alpha,\beta}(1-\tau)^{-(1+\gamma)\bar{\lambda}}$ and $\bar{\lambda}=2/p+n/\alpha$.

The rest of the proof is identical to the one for {\sc Lemma \ref{Exponent Reduction Time}}.

\section{Reverse Cauchy-Schwarz type inequality}
In this section, we will prove {\sc Proposition \ref{S JN}}, the reverse Cauchy-Schwartz type inequality. Recalling the following definition for all positive function $v>0$ and space-time regions $D_1$ and $D_2$,
$$\mathcal{F}[v, \alpha]_{D_1,D_2} := \smp{ \int_{D_1} v^{-\alpha}\;dxdt } \smp{ \int_{D_2}  v^{\alpha} \;dxdt },$$
we restate our goal here.
\begin{prop}
\label{S JN_copy}
Let $u$ be a non-negative super-solution of \eqref{basiceqn} in $[0,2]\times B_1$. Given $t\in(0,1)$, for every $\epsilon >0$, there exist  constants $\alpha_{\epsilon}$ and $K_{\epsilon}$ depending only on $n, \iota, \Lambda,t$ and  $\epsilon$ such that $\forall \mu>0$
\begin{equation}
\mp \bigp{ \mathcal{F}[u+\mu,\alpha_{\epsilon}]^{1/\alpha_{\epsilon}}_{D^+_0,D^-_0} > K_{\epsilon}  }  < \epsilon.
\end{equation}
Here $D^+_0=(2-t^2,2)\times B_t$ and $D^-_0=(0,t^2)\times B_t$.
\end{prop}

To better present our idea, we assume $t=1/2$ for now. Our proof will work for every $t\in(0,1)$ with minor adjustments and we will point out the differences after the proof.

We write $h^\mu=-\log(u+\mu)$ and we will prove a tail estimate for $\int_{D^+} e^{\nu h^\mu}dxdt\int_{D^-}e^{-\nu h^\mu}dxdt$ regardless of $\mu$ when $\nu$ is small. The proof of such estimate relies heavily on a variant of the parabolic John-Nirenberg inequality. To present the variant, we define a few collections of space-time rectangular regions as following.

We first create a large collection of cubes within $[0,2]\times B_1$ starting from $C_0=(0,2)\times B_{1/2}$.
We start by defining $\mathscr{C}_0=\{C_0\}$. For every cube $C$ assuming the form $(l-4s,l+4s)\times B_z(w)$, we write $C^+:=(l,l+4s)\times B_z(w)$, $C^-:=(l-4s,l)\times B_z(w)$, $D^+:=(l+3s,l+4s)\times B_z(w)$, $D^-:=(l-4s,l-3s)\times B_z(w)$, $I^+:=(l+2s,l+4s)\times B_z(w)$ and $I^-:=(l-4s,l-2s)\times B_z(w)$ as in Fig \ref{Fig:7}. In this way, we have $C_0^+,C_0^-,D_0^+,D_0^-,I_0^+$ and $I_0^-$ defined, and we define $\mathscr{C}^+_0,\mathscr{C}^-_0,\mathscr{D}^+_0,\mathscr{D}^-_0,\mathscr{I}^+_0$ and $\mathscr{I}^-_0$ as the collection made up of each of them, respectively.

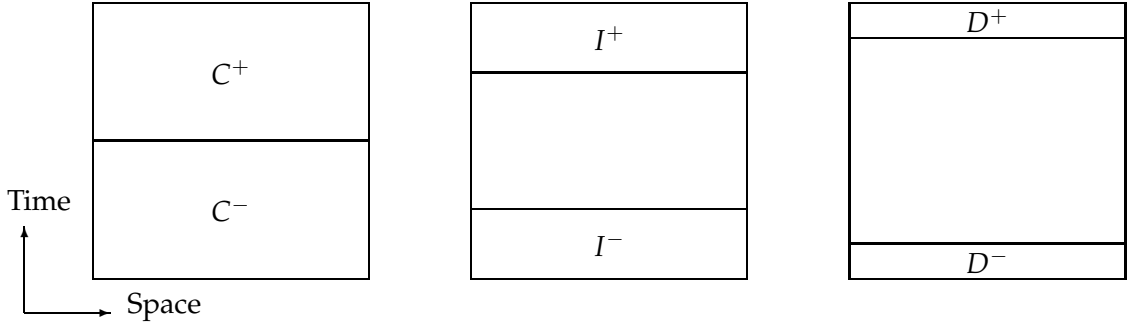
\begin{figure}[!htbp]
\setlength{\unitlength}{0.09in}
\centering
\begin{picture}(60,22)
\put(0,3){\framebox(16,8){$C^-$}}
\put(0,11){\framebox(16,8){$C^+$}}
\put(22,3){\framebox(16,4){$I^-$}}
\put(22,15){\framebox(16,4){$I^+$}}
\put(21.98,7){\line(0,1){8}}
\put(38.02,7){\line(0,1){8}}
\put(44,3){\framebox(16,2){$D^-$}}
\put(44,17){\framebox(16,2){$D^+$}}
\put(43.97,5){\line(0,1){12}}
\put(60.03,5){\line(0,1){12}}
\put(-4,1){\vector(0,1){5}}
\put(-4,1){\vector(1,0){5}}
\put(-5,7){Time}\put(2,1){Space}
\end{picture}
\caption{The whole cube is $C$, $C^\pm$ occupy the upper/lower half, $D^\pm$ occupy the upper/lower one eighth, and $I^\pm$ occupy the upper/lower quarter.}
\label{Fig:7}
\end{figure}

Now assume $\mathscr{C}_i$ has been defined for some $i$, we divide each cube in $\mathscr{D}^+_i$ and $\mathscr{D}^-_i$ into $4^{n+2}$ congruent pieces by dividing the time interval it spans into $16$ equal pieces and space interval in each dimension into $4$ equal pieces. For any one of these smaller cubes, if it comes from dividing some $D^+_i(k)$ in $\mathscr{D}_i^+$, we think of it as the $D^+$ of some $C$ and put that $C$ into $\mathscr{C}_{i+1}$; otherwise we think of it as the $D^-$ of some $C$ and put that $C$ into $\mathscr{C}_{i+1}$. After this is done for all the smaller cubes coming from the division, we define $\mathscr{C}^+_{i+1},\mathscr{C}^-_{i+1},\mathscr{D}^+_{i+1},\mathscr{D}^-_{i+1},\mathscr{I}^+_{i+1}$ and $\mathscr{I}^-_{i+1}$ as the collection of the respective cubes corresponding to the ones in $\mathscr{C}_{i+1}$. We repeat this process and define $\mathscr{C}(C_0)=\bigcup_{j=1}^\infty \mathscr{C}_j$.

We now define a new collection $\mathscr{C}'(C_0)$. The way to define it is almost identical to the process creating $\mathscr{C}(C_0)$. Again we start with $C_0$, but this time we proceed by dividing $I^\pm$ into $4^{n+2}$ congruent pieces instead of $D^\pm$. After constructing $\mathscr{C}'(C_0)$, we note that such construction can actually be done starting with any cube in $\mathscr{C}(C_0)$, and we write $\mathscr{C}':=\bigcup\limits_{C_j(i)\in\mathscr{C}(C_0)}\mathscr{C}'(C_j(i))$. For the convenience of later calculation, we re-arrange the labels in $\mathscr{C}'$ so any $C_m(k)$ from it has spatial radius $2^{-2m-1}$.

Our parabolic John-Nirenberg inequality takes the following form.
\begin{prop}\label{S JN full}
Assume $f$ is an $L^2$ function on $(0,2)\times B_1$. Suppose that we have a constant $A>0$ such that there exists $a_{C_j(i)}$ for every cube $C_j(i)\in\mathscr{C}'$ satisfying the following inequalities on the corresponding cubes $C^\pm_j(i)$,
\begin{equation}\label{S JN equation 1 copy}\frac{1}{|C^+_j(i)|}\int_{C^+_j(i)}\sqrt{(f(t,x)-a_{C_j(i)})^+}dxdt\leq A,\end{equation}
\begin{equation}\label{S JN equation 2 copy}\frac{1}{|C^-_j(i)|}\int_{C^-_j(i)}\sqrt{(a_{C_j(i)}-f(t,x))^+}dxdt\leq A.\end{equation}
Then there exist two positive constants $B$ and $b$ depending only on the dimension $n$ and $A$ such that for every $\alpha>0$,
$$|\bigp{(t,x)\in D_0^+|(f(t,x)-a_{C_0})^+>\alpha}|\leq Be^{-\frac{b\alpha}{A}}|D_0^+|,$$
$$|\bigp{(t,x)\in D_0^-|(a_{C_0}-f(t,x))^+>\alpha}|\leq Be^{-\frac{b\alpha}{A}}|D_0^-|.$$
Furthermore, for any $0<\nu<\frac{b}{A}$,
$$\int_{D_0^+}e^{\nu f}dxdt\int_{D_0^-}e^{-\nu f}dxdt\leq B^2\nu^2\abs{D_0^+}\abs{D_0^-}\smp{\int_0^\infty e^{(\nu-\frac{b}{A})\alpha}d\alpha}^2.$$
\end{prop}

\begin{remark}
\label{head count}
We note here that $\mathscr{D}_j^+$ is made up of $2^j\times4^{j(n+2)}$ elements with spatial radius $2^{-1}\times4^{-j}$. As for the collection $\mathscr{C}'$, we note that any cube in $\mathscr{C}'$ with spatial radius $2^{-1}\times4^{-j}$ either comes from dividing cubes with spatial radius $2^{-1}\times4^{-j+1}$ into $4^{n+2}$ pieces or lives in $\mathscr{C}(C_0)$. We denote by $x_j$ the number of cubes with spatial radius $2^{-1}\times4^{-j}$ in $\mathscr{C}'$, the previous observation gives $x_j=2\times4^{n+2}x_{j-1}+2^j\times4^{(n+2)j}$ which leads to $x_j\leq 4^{(n+3)j}$.
\end{remark}

The proof of this parabolic John-Nirenberg inequality is identical to the classical ones as in Fabes and Garofalo\cite{FG85} and Moser \cite{Moser64}. We therefore only sketch the proof and omit the details here.

The main tools in proving {\sc Proposition \ref{S JN full}} are the following two lemmas.

\begin{lem}\label{S JN Step 1}
Assume $f$ is an $L^2$ function on $(0,2)\times B_1$. Suppose we have a constant $A>0$ such that there exists $a_{C_j(i)}$ for every cube $C_j(i)\in\mathscr{C}$ satisfying the following two inequalities on the corresponding cubes $I^\pm_j(i)$,
\begin{equation}\label{SJN step 1 assumption 1}\frac{1}{|I^+_j(i)|}\int_{I^+_j(i)}(f(t,x)-a_{C_j(i)})^+dxdt\leq A,\end{equation}
\begin{equation}\frac{1}{|I^-_j(i)|}\int_{I^-_j(i)}(a_{C_j(i)}-f(t,x))^+dxdt\leq A.\end{equation}
Then there exist two positive dimensional constants $B$ and $b$ such that for every $\alpha>0$,
$$|\bigp{(t,x)\in D_0^+|(f(t,x)-a_{C_0})^+>\alpha}|\leq Be^{-\frac{b\alpha}{A}}|D_0^+|,$$
$$|\bigp{(t,x)\in D_0^-|(a_{C_0}-f(t,x))^+>\alpha}|\leq Be^{-\frac{b\alpha}{A}}|D_0^-|.$$
\end{lem}
\begin{lem}\label{S JN Step 2}
Assume $f$ is an $L^2$ function on $(0,2)\times B_1$. Suppose that we have a constant $A'>0$ such that there exists $a_{C_j(i)}$ for every cube $C_j(i)\in\mathscr{C}'$ satisfying the following inequalities on the corresponding cubes $C^\pm_j(i)$,
\begin{equation}\label{S JN equation 1}\frac{1}{|C^+_j(i)|}\int_{C^+_j(i)}\sqrt{(f(t,x)-a_{C_j(i)})^+}dxdt\leq A',\end{equation}
\begin{equation}\label{S JN equation 2}\frac{1}{|C^-_j(i)|}\int_{C^-_j(i)}\sqrt{(a_{C_j(i)}-f(t,x))^+}dxdt\leq A'.\end{equation}
Then there exists two positive dimensional constants $B'$ and $b'$ such that for every $\alpha>0$ and $C_m(k)\in\mathscr{C}(C_0)$, the following two inequalities are satisfied on the corresponding cubes $I_m^\pm(k)$,
$$|\bigp{(t,x)\in I_m^+(k)|(f(t,x)-a_{C_m(k)})^+>\alpha}|\leq B'e^{-b'(\frac{\alpha}{A'})^{\frac{1}{2}}}|I_m^+(k)|,$$
$$|\bigp{(t,x)\in I_m^-(k)|(a_{C_m(k)}-f(t,x))^+>\alpha}|\leq B'e^{-b'(\frac{\alpha}{A'})^{\frac{1}{2}}}|I_m^-(k)|.$$
\end{lem} 

These two lemmas are the exact copies of \cite[Theorem 1 and 2]{FG85} with the space-time rectangular regions used in the proofs specified. Therefore their proofs will not be included in our article. With these two lemmas, we can provide a short proof of {\sc Proposition \ref{S JN full}}.

\begin{proof}[Proof of Proposition \ref{S JN full}]
From {\sc Lemma \ref{S JN Step 2}}, we have on each $C_m(k)\in\mathscr{C}_0$,
$$\int_{I^+_m(k)}(f(t,x)-a_{C_m(k)})^+dxdt\leq\abs{I^+_m(k)}\int_0^\infty B'e^{-b'(\frac{\alpha}{A'})^{\frac{1}{2}}}d\alpha,$$
hence \eqref{SJN step 1 assumption 1} is satisfied. A similar argument shows that the other inequality in the assumptions of {\sc Lemma \ref{S JN Step 1}} is also satisfied with $A=\int_0^\infty B'e^{-b'(\frac{\alpha}{A'})^{\frac{1}{2}}}d\alpha$.

At this moment, {\sc Lemma \ref{S JN Step 1}} can be applied. We have for each $\nu<\frac{b}{A}$,
\begin{equation}\label{estimate S JN}
\int_{D_0^+}e^{\nu f}dxdt\leq\int_{D_0^+}e^{\nu a_{C_0}}e^{\nu (f-a_{C_0})^+}dxdt\leq Be^{\nu a_{C_0}}\nu|D_0^+|\int_0^\infty e^{(\nu-\frac{b}{A})\alpha}d\alpha,
\end{equation}
\begin{equation}\label{estimate S JN 1}
\int_{D_0^-}e^{-\nu f}dxdt\leq\int_{D_0^-}e^{-\nu a_{C_0}}e^{\nu (a_{C_0}-f)^+}dxdt\leq Be^{-\nu a_{C_0}}\nu|D_0^-|\int_0^\infty e^{(\nu-\frac{b}{A})\alpha}d\alpha.
\end{equation}
Therefore the proposition is proved.
\end{proof}

With {\sc Proposition \ref{S JN full}} in hand, our goal now is to find suitable $a_{C_j(i)}$s and $A$ satisfying the assumptions in the proposition for $f=h^\mu$. Since we have a stochastic perturbation term in \eqref{basiceqn}, we cannot expect an almost sure result with fixed $A$ and deterministic $a_{C_j(i)}$s. However, we can get an almost sure statement including a random perturbation, and then bound the perturbation on a large probability.

To state our results, a few extra notations need to be introduced. We write $r_j=2^{-1}\times4^{-j}$ for simplicity and pick a smooth cut-off function $\phi$ which is $1$ on $B_{1/2}$, is $0$ outside $B_{3/4}$ with convex level set, and is bounded between $[0,1]$. For any $C_j(i)\in\mathscr{C}'$, its spatial radius is $r_j$ and there exists $(s,x)$ such that
$$C_j(i)=(s-4r_j^2,s+4r_j^2)\times B_{r_j}(x).$$

On $C_j(i)$, denoting by $\phi_{B_{r_j}(x)}(y):=\phi((2r_j)^{-1}(y-x))$ the cut-off function scaled to $B_{r_j}(x)$, recalling $h^\mu=-\log(u+\mu)$ and introducing $|V(C_j(i),\phi)|=\int_{B_{3r_j/2}(x)}\phi^2_{B_{r_j}(x)}(y)dy$, we define
\begin{equation}\label{quadratic variation estimate}
\begin{cases}M_{C_j(i)}^\mu (t) :=\sum_i \int_{s}^{s+t} \frac{1}{|V(C_j(i),\phi)|}  \int_{ B_{ 3r_j/2}  (x)} \tilde{g}^{\mu}_i (\tau,y,u;\omega)  \phi_{B_{r_j}(x)}^2(y)\;dy dw^{i}_{\tau},\\
H_{C_j(i)}^\mu (t) := \frac{1}{|V(C_j(i),\phi)|}  \int_{ B_{ 3r_j/2}(x)} h^\mu (t +s,y;\omega) \phi_{B_{r_j}(x)}^2(y) \;dy,\end{cases}
\end{equation}
where $\tilde{g}^{\mu}_i(t,x,u;\omega)  :=g_i(t,x,u;\omega)  (u+\mu)^{-1}.$
\begin{remark}\label{S JN Step 3 bound on M t}
The quadratic variation process $\ip{M_{C_j(i)}^\mu}_t$ is bounded by constant times of $t$.
\end{remark}

We have the following almost sure result.
\begin{lem}
\label{s local bound}
Let $u$ be a non-negative super-solution to \eqref{basiceqn} in $[0,2]\times B_1$. There exists a constant $\bar{A}$ depending only $n, \iota, \Lambda$ such that for every $C_j(i) \in \mathscr{C}'$, we can find a random variable $a_{C_j(i)}$ satisfying
\[
\frac{1}{ \abs{C_j^+(i)} } \int_{C^+_j(i) }  \sqrt{ \smp{ h^\mu - M^\mu_{C_j(i)} - a_{C_j(i)} }^+} \;dxdt \leq \bar{A} \quad \text{a.s.}
\]
and
\[
\frac{1}{ \abs{C^-_j(i)} } \int_{C^-_j(i) }  \sqrt{ \smp{  M^\mu_{C_j(i)} + a_{C_j(i)}- h^\mu }^+} \;dxdt \leq \bar{A} \quad \text{a.s..}
\]
\end{lem}
\begin{proof}
We will use $(\cdot,\cdot)$ to denote the inner product on $\mr^n$. By direct calculation, $h^\mu$ is a sub-solution of
\begin{equation}
\label{M-J-N log eq}
dh^\mu  = \divg( \mA \nabla h^\mu) dt- (\mA \nabla h^\mu, \nabla h^\mu)dt + \tilde{f}^\mu \;dt+ \tilde{g}_i^{\mu} dw^i_t
\end{equation}
with
\[
\tilde{f}^\mu (t,x,u;\omega) = f(t,x,u;\omega)(u+\mu)^{-1} + 2^{-1}\abs{\tilde{g}^{\mu} }^2_{\ell^2}.
\]

Fixing any $C_j(i) \in \mathscr{C}' $, by testing \eqref{M-J-N log eq} with $\phi_{B_{r_j}(x)}^2(y)$, we have for any $t_1$ and $t_2$ such that $s-4r_j^2< t_1\leq t_2< s+4r^2_j$,
\begin{align*}
&\int_{B_1} h^\mu(t_2) \phi_{B_{r_j}(x)}^2 \;dy - \int_{B_1} h^\mu(t_1) \phi_{B_{r_j}(x)}^2 \;dy   + \int_{t_1}^{t_2} \int_{B_1}\smp{\mA\nabla h^\mu,\nabla h^\mu}\phi_{B_{r_j}(x)}^2\;dyd\tau   \\
&\leq-\int_{t_1}^{t_2}\ip{\mA\nabla h^\mu, \nabla(\phi_{B_{r_j}(x)}^2)}\;d\tau + \int_{t_1}^{t_2}\int_{B_1}\tilde{f}^\mu\phi^2_{B_{r_j}(x)}\;dyd\tau+ \sum_i \int_{t_1}^{t_2} \int_{B_1} \tilde{g}_k^{\mu}\phi_{B_{r_j}(x)}^2 \;dy dw^k_\tau.
\end{align*}

Applying the Cauchy-Schwarz inequality to the first term on the right hand side, the above inequality implies
\begin{align*}
\int_{B_{3r_j/2}(x)} h^\mu(t_2) \phi_{B_{r_j}(x)}^2 \;dy &- \int_{B_{3r_j/2}(x)} h^\mu(t_1) \phi_{B_{r_j}(x)}^2 \;dy   + \frac{1}{2}\int_{t_1}^{t_2} \int_{B_{3r_j/2}(x)}\smp{\mA\nabla h^\mu,\nabla h^\mu}\phi_{B_{r_j}(x)}^2\;dyd\tau   \\
&\leq2\int_{t_1}^{t_2}\ip{\nabla\phi_{B_{r_j}(x)},\mA\nabla\phi_{B_{r_j}(x)}}\;dyd\tau\\&\quad+ \int_{t_1}^{t_2}\int_{B_{3r_j/2}(x)}\tilde{f}^\mu\phi^2_{B_{r_j}(x)}\;dyd\tau+ \sum_i \int_{t_1}^{t_2} \int_{B_{3r_j/2}(x)} \tilde{g}_k^{\mu}\phi_{B_{r_j}(x)}^2 \;dydw^k_\tau.
\end{align*}

Using the uniform ellipticity of $\mA$ and the growth bound of $f$, we obtain for some positive constants $C_1,C_2$ and $C_3$,
\begin{align*}
\int_{B_{3r_j/2}(x)} h^\mu(t_2) \phi_{B_{r_j}(x)}^2 \;dy - &\int_{B_{3r_j/2}(x)} h^\mu(t_1) \phi_{B_{r_j}(x)}^2 \;dy   + C_1^{-1} \int_{t_1}^{t_2} \norm{\phi_{B_{r_j}(x)} \nabla h^\mu  }_{2,B_{3r_j/2}(x)}^2 \;ds   \\
&\leq  (C_2r_j^{-2}+C_3)|V(C_j(i),\phi)|(t_2-t_1)+ \sum_i \int_{t_1}^{t_2} \int_{B_{3r_j/2}(x)} \tilde{g}_k^{\mu}\phi_{B_{r_j}(x)}^2 \;dy \;dw^k_\tau.
\end{align*}
Dividing the above inequality by $|V(C_j(i),\phi)|$ and applying a weighted Poincar\'e's inequality (see \cite[Lemma 3, Page 120]{Moser64}) to the third term on the left hand side, we have the following differential inequality for a different $C_1$,
\begin{equation}
\label{sd ineq}
d H_{C_j(i)}^\mu + \smp{ \frac{C_1^{-1}r_j^{-2}}{|V(C_j(i),\phi)| } \int_{ B_{3r_j/2} (x)} ( h^\mu  - H_{C_j(i)}^\mu)^2\phi_{B_{r_j}(x)}^2dy} dt \leq (C_2r_j^{-2}+C_3)dt + d M_{C_j(i)}^\mu.
\end{equation}
We emphasize that the above inequality should be interpreted in the integral form, and this rule applies to all the differential inequalities below.

We define the following two stochastic processes
$$z(t,y):= h^\mu(s+t,y) - M_{C_j(i)}^\mu (t)  - (C_2r_j^{-2}+C_3)t-H_{C_j(i)}^\mu ( 0)$$
and
$$Z(t) := H_{C_j(i)}^\mu ( t)-H_{C_j(i)}^\mu ( 0) - M_{C_j(i)}^\mu (t) -(C_2r_j^{-2}+C_3)t.$$

The inequality \eqref{sd ineq} is equivalent to
\begin{equation}
\label{M-J-N p 1}
dZ + \frac{C_1^{-1}r_j^{-2}}{|V(C_j(i),\phi)| } \int_{ B_{3r_j/2 (x)}} (z - Z)^2\phi_{B_{r_j}(x)}^2  \;dy dt \leq 0 ,\quad  Z(0) =0.
\end{equation}

We now extract a growth bound for the level set of $z$ from \eqref{M-J-N p 1}. The inequality implies immediately for all non-negative $t$,
\[
Z (t) \leq 0, \quad \mp \text{ almost surely.}
\]

For arbitrary $a \geq1$, we write
$E_a (t):= \bigp{ y \in  B_{r_j} (x):  z (t,y) > a }.$
On $E_a (t)$, we have $0 < (a - Z) \leq (z - Z),$
thus $(a -Z)^2\abs{E_a (t)}  \leq \int_{ B_{3r_j/2} (x)} \smp{z - Z} ^2\phi_{B_{r_j}(x)}^2 \;dy .$
Then it follows from \eqref{M-J-N p 1} that
$$\frac{1}{(a - Z)^2 }d Z   + \frac{C_1^{-1}r_{j}^{-2} }{  |V(C_j(i),\phi)|}\abs{E_a (t) } \;dt  \leq 0.$$

By It\^o's formula,
$d \frac{1}{a - Z} =  \frac{1}{ (a - Z)^2} dZ + \frac{1}{ (a - Z)^3 } \;d \ip{M_{C_j(i)}^\mu }_t,$
so we have
$$\frac{C_1^{-1}r_{j}^{-2} }{  |V(C_j(i),\phi)|}\abs{E_a (t) } \;dt \leq  -d\frac{1}{a-Z}+\frac{1}{ (a- Z)^3} d\ip{M_{C_j(i)}^\mu }_t.$$

Using the above inequality in its integral form between $0$ and $4r_{j}^2$ and applying the quadratic variation bound on $M^\mu $, we obtain with the fact $r_j\leq1/2$,
$$\abs{\bigp{(t,y)\in C_j^+(i) :  z(t-s,y) > a}} \leq  \frac{C_4}{ a} \abs{C_j^+(i)}$$
for a constant $C_4$.

Recalling the expression of $z$, we have, for each $a \geq 1$
\[
\abs{\bigp{ (t, y) \in C_j^+(i) :  h^\mu(t,y )  -M_{C_j(i)}^\mu (t-s)  - (C_2r_j^{-2}+C_3)(t-s)  -H_{C_j(i)}^\mu ( 0)> a}} \leq  \frac{C_4}{a}\abs{C_j^+(i)}.
\]

Therefore, if we choose $a_{C_j(i)} = H_{C_j(i)}^\mu ( 0) $ and note the fact $r_j\leq1/2$ again,
\begin{align*}
&\;\quad\int_{C^+_j(i)} \sqrt{ \smp{ h^\mu(t) - M^\mu_{C_j(i)}(t-s) - a_{C_j(i)} }^+}\;dydt \\&\leq  \int_0^{4r_j^2}\int_{B_{r_j}(x)} \sqrt{ \smp{ h^\mu(t+s) - M^\mu_{C_j(i)}(t) - a_{C_j(i)} -(C_2r_j^{-2}+C_3)t}^+}+ \sqrt{(C_2r_j^{-2}+C_3)t} \;dydt\\
&\leq \int_{C^+_j(i), z\leq1} \sqrt{z^+(t-s)} \;dydt +\int_{C^+_j(i),z>1}\sqrt{z(t-s)}\;dydt+ 2\sqrt{C_2+C_3}\abs{C_j^+(i)}\\
&\leq  \int_1^{\infty} \frac{C_4\abs{C_j^+(i)}}{a} d\sqrt{a}  + (2+2\sqrt{C_2+C_3})\abs{C_j^+(i)} \leq A\abs{C_j^+(i)}.
\end{align*}

The other inequality can be proved by a completely symmetric procedure from \eqref{sd ineq} using $h(s-t,x)$ instead of $h(s+t,x)$ and this completes the proof for the lemma.
\end{proof}

Our next step is to bound the random perturbation term $M^\mu$ in {\sc Lemma \ref{s local bound}}.
\begin{lem}\label{S JN Step 3}
Let $u$ be a non-negative super-solution to \eqref{basiceqn} in $[0,2]\times B_1$. For every $\epsilon >0$, there exist a constant $A$ depending only on $(n, \iota, \Lambda, \epsilon)$ and random variables $a_j(i)$s such that for all $\mu>0$, inequalities \eqref{S JN equation 1 copy} and \eqref{S JN equation 2 copy} are satisfied for $f=h^\mu$ and any $C_j(i)\in\mathscr{C}'$ on a set of at least probability $1-\epsilon$.
\end{lem}
\begin{proof}
Since {\sc Lemma \ref{s local bound}} gives two upper bounds almost surely, we only need to uniformly bound $M^\mu_{C_j(i)}(t)$ regardless of $\mu$.

We recall from {\sc Remark \ref{S JN Step 3 bound on M t}} that there is a constant $J$ satisfying $\ip{M^\mu_{C_j(i)}}_t\leq Jt$.
Therefore for a large number $L>0$ and any $C_j(i)\in\mathscr{C}'$, by \cite[Lemma 3.1]{HWW14} we obtain,
$$\mp\bigp{\sup_{0\leq t\leq 4r_j^2}|M^\mu_{C_j(i)}(t)|\geq L,\; \ip{M^\mu_{C_j(i)}}_{4r_j^2}\leq 4r_j^2J}\leq \exp\bigp{-L^2/(16r_{j}^2J)}.$$
Since the second condition is always true, we have $$\mp\bigp{\sup_{0\leq t\leq 4r_{j}^2}|M^\mu_{C_j(i)}(t)|\geq L}\leq \exp\bigp{-L^2/(16r_{j}^2J)}.$$

This inequality and {\sc Lemma \ref{s local bound}} tell us for some constant $\bar{A}$, $$\mp\bigp{\frac{1}{|C^+_j(i)|}\int_{C^+_j(i)}\sqrt{(h^\mu-a_{C_j(i)})^+}dxdt> \bar{A}+\sqrt{L}}\leq \exp\bigp{-L^2/(16r_{j}^2J)},$$
and a similar argument provides $$\mp\bigp{\frac{1}{|C^-_j(i)|}\int_{C^-_j(i)}\sqrt{(a_{C_j(i)}-h^\mu)^+}dxdt> \bar{A}+\sqrt{L}}\leq \exp\bigp{-L^2/(16r_{j}^2J)}.$$

In {\sc Remark \ref{head count}}, we have denoted by $x_j$ the number of cubes with spatial radius $r_j$. Considering the event $\Theta(L)$ that there exists at least one of the $C_j(i)$s such that \eqref{S JN equation 1 copy} or \eqref{S JN equation 2 copy} fails with $A+\sqrt{L}$, we have
\begin{equation}\label{S JN final estimate}
\begin{split}
\mp\bigp{\Theta(L)}\leq\sum_{j=0}^\infty2x_j\exp\bigp{-L^2/(16r_{j}^2J)}
\leq2\sum_{j=0}^\infty4^{(n+3)j}\exp\bigp{-16^jL^2/(4J)}.
\end{split}
\end{equation}

Therefore we can choose a sufficiently large $L$ to make $\mp\bigp{\Theta(L)}\leq\epsilon$. This concludes our proof of the proposition with $A=\bar{A}+\sqrt{L}$.
\end{proof}

Now we can finish the proof of {\sc Proposition \ref{S JN_copy}}.
\begin{proof}[Proof of Proposition \ref{S JN_copy}]
For each $\epsilon>0$, {\sc Lemma \ref{S JN Step 3}} provides us that on a set $\Omega'\subseteq\Omega$ with probability at least $1-\epsilon$, \eqref{S JN equation 1 copy} and \eqref{S JN equation 2 copy} hold for $h^\mu=-\log(u+\mu)$ with some random variables $a_{C_j(i)}$s and a constant $A$ depending on $\epsilon$ on all $C_j(i)\in\mathscr{C}'$.

Applying {\sc Proposition \ref{S JN full}} with $f=h^\mu$ on the set $\Omega'$, we have for $\nu=\frac{b}{2A}$,
$$\mathscr{F}[u+\mu,\nu]_{D^+_0,D^-_0} \leq B^2\nu^2\abs{D_0^+}\abs{D_0^-}\smp{\int_0^\infty e^{-\frac{b\alpha}{2A}}d\alpha}^2$$
on $\Omega'$. The proposition is proved.
\end{proof}

\begin{remark}
So far we have proved {\sc Proposition \ref{S JN_copy}} for $t=1/2$. For other values of $t$, we need to make the following changes.
\begin{itemize}
\item $C_0$ will be changed to $(0,2)\times B_t$.
\item The relative positions of $C^\pm$, $D^\pm$, $I^\pm$ and $C$ in the division constructions of the cube collections will not change much and $C^\pm$ will still take the upper/lower halves of $C$. However, $D^\pm$ need to occupy the upper(lower) $t^2/2$ portions and $I^\pm$ need to occupy the upper(lower) $t/2$ portions.
\item The division processes mentioned above also need to be finer. We need to choose an integer $\zeta$ sufficiently large and divide the cubes into $\zeta^{n+2}$ pieces instead of $4^{n+2}$. The criteria for the choice of $\zeta$ is to allow the proofs of {\sc Lemma \ref{S JN Step 1}} and {\sc Lemma \ref{S JN Step 2}} to go through.
\item The smooth cut-off function in the third step needs to be $1$ on $B_t$ and vanishes outside of $B_{(1+t)/2}$.
\end{itemize}
The rest of the proof for the cases $t\neq1/2$ is a verbatim repetition of the proofs in this and the next sections.
\end{remark}

\section{Acknowledgement}
The project was initially started by the author and Doctor Yu Wang (currently in Goldman Sachs) in $2014$ upon the completion of \cite{HWW14}. Although the collaboration ended after the departure of Yu Wang, the discussion with him has helped clarify many confusions. His contribution to this project is greatly appreciated.

\end{document}